\documentclass[english]{article}
\usepackage[T1]{fontenc}
\usepackage[latin9]{inputenc}
\usepackage{geometry}
\usepackage{color}
\usepackage{mathrsfs}
\usepackage{amsthm}
\usepackage{amsmath}
\usepackage{amssymb}

\makeatletter

\numberwithin{equation}{section}
\theoremstyle{plain}
\newtheorem{thm}{\protect\theoremname}[section]
  \theoremstyle{definition}
  \newtheorem{defn}[thm]{\protect\definitionname}
  \theoremstyle{plain}
  \newtheorem{prop}[thm]{\protect\propositionname}
  \theoremstyle{plain}
  \newtheorem{lem}[thm]{\protect\lemmaname}
  
  \newcounter{casectr}
  
  \theoremstyle{remark}
  \newtheorem{rem}[thm]{\protect\remarkname}
  \theoremstyle{plain}
  \newtheorem{cor}[thm]{\protect\corollaryname}
  \theoremstyle{remark}
  \newtheorem*{rem*}{\protect\remarkname}

\usepackage{hyperref}
\usepackage{graphicx}
\usepackage{fancyhdr}
\usepackage{subcaption}

\makeatother

\usepackage{babel}
  \providecommand{\casename}{Case}
  \providecommand{\corollaryname}{Corollary}
  \providecommand{\definitionname}{Definition}
  \providecommand{\lemmaname}{Lemma}
  \providecommand{\propositionname}{Proposition}
  \providecommand{\remarkname}{Remark}
\providecommand{\theoremname}{Theorem}

\begin{document}

\title{Extending generalized spin representations}

\author{Robin Lautenbacher and Ralf K\"ohl}

\maketitle

\abstract

We revisit the construction of higher spin representations by Kleinschmidt and Nicolai for $E_{10}$, generalize it to arbitrary simply laced types, and provide a coordinate-free approach to the $\frac{3}{2}$-spin and $\frac{5}{2}$-spin representations. Moreover, we discuss the relationship between our findings and the representation theory of $\mathrm{Sym}_3$ pointed out to us by Levy.

\section{Introduction}

Generalized spin representations of the maximal compact subalgebra of the split real Kac--Moody algebra of type $E_{10}$ have been introduced in \cite{Damour}, \cite{Buyl} and generalized to arbitrary symmetrizable types in \cite{Hainke}.  The purpose of this note is to revisit some of the higher spin representations of type $E_{10}$ studied in \cite{AN1}, notably $\frac{3}{2}$-spin and $\frac{5}{2}$-spin, generalize these to arbitrary simply laced types, and propose a coordinate-free approach which we carry out for $\frac{3}{2}$-spin and $\frac{5}{2}$-spin.  

Our main result is the following coordinate-free extension of generalized spin representations:

\medskip
\noindent
{\bf Theorem.} {\em Let $\mathfrak{g}$ be a simply laced split real Kac--Moody algebra, let $\mathfrak{h}$ be a Cartan subalgebra of $\mathfrak{g}$, let $\lambda$ be the set consisting of the simple roots of $\mathfrak{g}$ and roots that are sums of two distinct simple roots, let $\mathfrak{k}$ be the maximal compact subalgebra of $\mathfrak{g}$, and let $\left(\cdot\vert\cdot\right)$ denote
the induced invariant bilinear form on $\mathfrak{h}^{\ast}$. A map $X : \lambda\rightarrow\mathrm{End}\left(V\right)$ satisfying the following (anti-)commutator relations for all $\alpha,\beta\in\lambda$ 
\begin{eqnarray*}
\left[X(\alpha), X(\beta)\right] & = & 0\quad\qquad\quad\quad\,\text{if }\left(\alpha\vert\beta\right)=0\\
\left\{ X(\alpha),X(\beta)\right\}  & = & X(\alpha\pm\beta)\quad\text{if $\left(\alpha\vert\beta\right)=\mp1$ and $\alpha \pm \beta \in \lambda$}
\end{eqnarray*}
provides
a finite-dimensional representation $\sigma$ of $\mathfrak{k}$
via the assignment $$\sigma\left(X_{i}\right):=X\left(\alpha_{i}\right)\otimes\Gamma\left(\alpha_{i}\right)$$ on the Berman generators $X_{1},\dots,X_{n}$ of $\mathfrak{k}$, where the $\Gamma(\alpha_i)$, $1 \leq i \leq n$ are the anti-symmetric real matrices from \eqref{remarkgamma} induced by the generalized spin representation of $\mathfrak{k}$.

Define
$X_{\frac{3}{2}} : \Delta^{\mathrm{re}}\rightarrow\mathrm{End}\left(\mathfrak{h}^{\ast}\right)$
via 
\begin{eqnarray*}
\alpha\mapsto X_{\frac{3}{2}}(\alpha) & := & -\alpha\left(\alpha\vert\cdot\right)+\frac{1}{2}\mathrm{id}_{\mathfrak{h}^{\ast}}.
\end{eqnarray*}
Moreover, for $\alpha \in \Delta^{\mathrm{re}}$ let $\pi_{\alpha}:=\alpha\left(\alpha\vert\cdot\right)\in\mathrm{End}\left(\mathfrak{h}^{\ast}\right)$ and
define
$X_{\frac{5}{2}} : \Delta^{\mathrm{re}}\rightarrow\mathrm{End}\left(\mathrm{Sym}^2(\mathfrak{h}^{\ast})\right)$
via 
\begin{eqnarray*}
\alpha\mapsto X_{\frac{5}{2}}(\alpha) & := & \pi_{\alpha}\otimes\pi_{\alpha}-\left(\pi_{\alpha}\otimes \mathrm{id}_{\mathfrak{h}^{\ast}}+\mathrm{id}_{\mathfrak{h}^{\ast}}\otimes\pi_{\alpha}\right)+\frac{1}{2}\mathrm{id}_{\mathfrak{h}^{\ast}}\otimes \mathrm{id}_{\mathfrak{h}^{\ast}}. 
\end{eqnarray*}
Then $X_{\frac{3}{2}}$ and  $X_{\frac{5}{2}}$ satisfy the above equalities for all real roots $\alpha$, $\beta$
with $\left(\alpha\vert\beta\right)\in\left\{ 0,\pm1\right\}$ and
thus each provides a representation $\sigma$ of $\mathfrak{k}$.
}

\medskip
\noindent
The results for this note have been obtained during and shortly after the first author's MSc thesis project \cite{Lautenbacher} in mathematics.
It would be interesting to understand how these representations decompose into irreducible components. We refer to \cite{AN1}, \cite{AN2} for some investigations in this direction using coordinates.

\medskip
Paul Levy pointed out to us that both assignments $X_{\frac{3}{2}}$ and $X_{\frac{5}{2}}$ are of the form $$X(\alpha) := \rho(s_\alpha) - \frac{1}{2} \mathrm{id}$$ where $\rho(s_\alpha)$ denotes the natural reflection action of the fundamental generator $s_{\alpha}$ induced on $\mathfrak{h}^*$, resp.\ $\mathrm{Sym}^2(\mathfrak{h}^{\ast})$.
For simple roots $\alpha$, $\beta$ forming a subdiagram of type $A_2$ one obtains the equivalence $$\left\{ X(\alpha),X(\beta)\right\}  =  X(\alpha\pm\beta) \quad\quad \Longleftrightarrow \quad\quad \rho(s_\alpha s_\beta s_\alpha) - \rho(s_\alpha s_\beta) - \rho(s_\beta s_\alpha) + \rho(s_\alpha) + \rho(s_\beta) - \mathrm{id} = 0.$$
Among the irreducible representations of $\mathrm{Sym}_3$, the trivial and the geometric representations satisfy the above identity, whereas the sign representation does not.
One in fact arrives at a characterization of those representations $\rho : W \to \mathrm{GL}(V)$ of the Weyl group $W$ of $\mathfrak{g}$ that can be used for extending generalized spin representations via the assignment $X(\alpha) := \rho(s_\alpha) - \frac{1}{2} \mathrm{id}$: exactly those whose restrictions to any standard subgroup $\mathrm{Sym}_3 \cong \langle s_\alpha, s_\beta \rangle \leq W$ (where $\alpha$, $\beta$ are adjacent simple roots of $\mathfrak{g}$) do not contain a sign representation as an irreducible component will do.

Since neither of the given $W$-modules $\mathfrak{h}^*$ and $\mathrm{Sym}^2(\mathfrak{h}^{\ast})$ contain a $\mathrm{Sym}_3$-sign representation, they both can be used for extending generalized spin representations. 
The module $\mathrm{Sym}^3(\mathfrak{h}^{\ast})$ on the other hand does contain a sign representation and so the $\frac{7}{2}$-spin representations discussed in \cite{AN1}, \cite{AN2} still remain elusive.

\medskip
Moreover, note that a map $X : \lambda \to \mathrm{End}(V)$ as in the statement of the Theorem naturally extends to the set of all those positive real roots that can be written as iterated sums of simple roots such that each partial sum itself is a positive real root. It is well-known that in the finite-dimensional situation this set equals the set of all positive (real) roots; in the simply-laced affine case it can be shown that this set also equals the set of all positive real roots (cf.\ \cite{Lautenbacher}). To the best of our knowledge the question what this set looks like in general is open.

\medskip
Our note contains several redundancies. First, we reproduce the method to obtain extensions of generalized spin representations of $E_{10}$ and its application to $\frac{3}{2}$ and $\frac{5}{2}$-spin representations proposed by Kleinschmidt and Nicolai in order to make their work \cite{AN1}, \cite{AN2} accessible to a wider mathematical audience and to point out that their approach actually works for any simply-laced Dynkin diagram. Second, we propose and apply our own coordinate-free method. Third, we interpret our findings in terms of $\mathrm{Sym}_3$-representation theory based on Levy's observations. This organization of our note leads to various existence proofs of $\frac{3}{2}$ and $\frac{5}{2}$-spin representations and to a wealth of starting points for further investigation.

\medskip
\noindent
    {\bf Acknowledgements.} The first author thanks the Albert Einstein Institute in Golm for the hospitality in 2017. The second author gratefully acknowledges partial support from DFG via the project KO4323/13. The second author also thanks the Albert Einstein Institute in Golm for the hospitality in 2015 and in particular Axel Kleinschmidt and Hermann Nicolai for various discussions concerning the contents of \cite{AN1}.
The authors moreover thank Paul Levy for very valuable comments on a preliminary version of this note and for pointing out the relationship of their findings to the representation theory of $\mathrm{Sym}_3$.

\section{Generalized $\frac{1}{2}$-spin representations}

Recall the notion of a Kac--Moody algebra from \cite{Kac}.
Let $A$ be a symmetrizable
generalized Cartan matrix and $\left(\mathfrak{h},\Pi,\Pi^{v}\right)$
be a realization of $A$ over $\mathbb{R}$ so that for $\mathfrak{h}_\mathbb{C}:=\mathfrak{h} \otimes_{\mathbb{R}}\mathbb{C}$
the triple $\left(\mathfrak{h}_\mathbb{C},\Pi,\Pi^{v}\right)$ is a realization
of $A$ over $\mathbb{C}$. Let $\overline{\cdot} : \mathbb{C} \to \mathbb{C}$ be complex conjugation and denote by $\omega_{0}$ the $\overline{\cdot}$-semilinear involution on the complex Kac--Moody algebra $\mathfrak{g}_{\mathbb{C}}(A)$
determined by 
\[
\omega_{0}\left(e_{i}\right)=-f_{i}\ ,\ \omega_{0}\left(f_{i}\right)=-e_{i}\ ,\ \omega_{0}\left(h\right)=-h\ \forall\, h\in\mathfrak{h}_{\mathbb{R}}.
\]
Call $\omega_{0}$ the \textit{compact involution} of $\mathfrak{g}_{\mathbb{C}}(A)$
and $\mathfrak{k}_{\mathbb{C}}(A):=\text{Fix }\omega_{0}$ the \textit{maximal compact subalgebra} of $\mathfrak{g}_{\mathbb{C}}(A)$.

Let $\mathfrak{g}(A)$ be the split real form of $\mathfrak{g}_{\mathbb{C}}(A)$, i.e., the real Kac--Moody algebra obtained as the fixed points of complex conjugation $\overline{\cdot}$ acting naturally on the complex vector space underlying $\mathfrak{g}_{\mathbb{C}}(A)$. Let $\omega_{\mathbb{C}}$ and
$\omega$ denote the Chevalley involutions on these Kac--Moody
algebras and let $\omega_{0}$ be the compact involution on $\mathfrak{g}_{\mathbb{C}}(A)$.
Then one has 
\[
\mathfrak{g}_{\mathbb{C}}(A)\supset\text{Fix }\omega_{0}\cong\text{Fix }\omega\oplus i \omega_{-1},
\]
where $\omega_{-1}$ denotes the $-1$ eigenspace of $\omega$ on $\mathfrak{g}(A)$.
%
%
The fixed point subalgebra $\mathfrak{k}(A)=\text{Fix }\omega$
is called the \textit{maximal compact subalgebra} of $\mathfrak{g}(A)$.
%

A Kac--Moody algebra $\mathfrak{g}(A)$ is called \textsl{simply laced}
if its generalized Cartan matrix contains only entries which are $0$
or $-1$ on the off-diagonal.

\begin{thm} \label{gensandrels}
\label{cor:adaption of Berman pt 1}Let $\mathfrak{g}(A)$ be a simply
laced real Kac--Moody algebra and $\mathfrak{k}$ its maximal compact subalgebra. Then $\mathfrak{k}$
is isomorphic to the free Lie algebra over $\mathbb{R}$ of generators
$X_{1},\dots,X_{n}$ modulo the ideal generated by the relations
\begin{eqnarray*}
\left[X_{i},\left[X_{i},X_{j}\right]\right] & =-X_{j} & \ \text{if }a_{ij}=-1\\
\left[X_{i},X_{j}\right] & =\,0 & \ \text{if }a_{ij}=0\ 
\end{eqnarray*}
via the isomorphism given by
\[
X_{i}\mapsto e_{i}-f_{i}.
\]
\end{thm}
\begin{proof}
See \cite[Theorem~1.3]{Hainke}.
\end{proof}

\begin{defn}
\label{Def:gen spin rep for simply laced case}A representation $\rho : \mathfrak{k}\rightarrow\mathrm{End}\left(\mathbb{C}^{s}\right)$
is called a \textsl{generalized spin representation} if for the generators
$X_{1},\dots,X_{n}$ of $\mathfrak{k}$ one has
\[
\rho\left(X_{i}\right)^{2}=-\frac{1}{4}\mathrm{id}_{s}\ \forall\, i=1,\dots,n\,.
\]
\end{defn}

\begin{prop}
\label{prop:properties of gen spin reps in simply laced case}Let
$\rho : \mathfrak{k}\rightarrow\mathrm{End}\left(\mathbb{C}^{s}\right)$
be a generalized spin representation and denote by $[A,B] := AB-BA$ the commutator and by $\left\{ A,B\right\} :=AB+BA$
the anti-commutator. Then for $1\leq i\neq j\leq n$ one has
\begin{eqnarray*}
\left[\rho\left(X_{i}\right)\,,\,\rho\left(X_{j}\right)\right]=0 & \text{if }a_{ij}=0 & \Longleftrightarrow\ (i,j) \text{ do not form an edge of the Dynkin diagram}\\
\left\{ \rho\left(X_{i}\right)\,,\,\rho\left(X_{j}\right)\right\} =0 & \text{if }a_{ij}=-1 & \Longleftrightarrow\ (i,j) \text{ form an edge of the Dynkin diagram}.
\end{eqnarray*}
\end{prop}
\begin{proof}
If $(i,j)$ do not form an edge of the Dynkin diagram then $a_{ij}=0$ and so $\left[X_{i},X_{j}\right]=0$
according to Theorem~\ref{cor:adaption of Berman pt 1} which is
carried over to $\text{End}\left(\mathbb{C}^{s}\right)$, since $\rho$
is a homomorphism. If $(i,j)$ form an edge, which is to say $a_{ij}=-1$,
then by Theorem~\ref{cor:adaption of Berman pt 1} one has
\[
\left[X_{i},\left[X_{i},X_{j}\right]\right]=-X_{j}
\]
and setting $A=\rho\left(X_{i}\right)$, $B=\rho\left(X_{j}\right)$
one computes in $\text{End}\left(\mathbb{C}^{s}\right)$
\begin{eqnarray*}
\left[A,\left[A,B\right]\right] & = & -B\\
\Leftrightarrow\ A^{2}B-ABA-ABA+BA^{2} & = & -B\\
\Leftrightarrow\ -\frac{1}{4}B-2ABA-\frac{1}{4}B & = & -B\\
\Leftrightarrow\ -2ABA & = & -\frac{1}{2}B\ \quad\quad\quad \vert\ \cdot A\ \text{from the right}\\
\Leftrightarrow\ \frac{1}{2}AB & = & -\frac{1}{2}BA\\
\Leftrightarrow\ AB+BA & = & 0.
\end{eqnarray*}
Note that multiplication with $A$ preserves equivalence
because $A$ is invertible, since $A^{-1}=-4A$.
\end{proof}

\begin{cor}
\label{cor:Sufficient conditions for gen spin rep}Given matrices
$A_{1},\dots,A_{n}\in\mathbb{C}^{s\times s}$with 
\begin{eqnarray*}
(i)\ A_{i}^{2} & = & -\frac{1}{4}\mathrm{id}_{s},\\
(ii)\ \left[A_{i},A_{j}\right] & = & 0, \text{ if $(i,j)$ do not form an edge of the Dynkin diagram},\\
(iii)\ \left\{ A_{i},A_{j}\right\}  & = & 0, \text{ if $(i,j)$ form an edge of the Dynkin diagram},
\end{eqnarray*}
 the extension of the map $X_{i}\mapsto A_{i}$ defines a generalized
 spin representation $\rho$ from $\mathfrak{k}$ on $\mathbb{C}^{s}$.
\end{cor}

\begin{proof}
$(i)$ is a necessary condition by the definition of spin representations.
Assertion $(ii)$ ensures that the commutation relations between $X_{i},X_{j}$
are respected by $\rho$ if $(i,j)$ do not form an edge, because in this case
$\left[X_{i},X_{j}\right]=0$. Finally, $(iii)$ ensures that for
$a_{ij}\neq0$ the relation 
\[
\left[X_{i},\left[X_{i},X_{j}\right]\right]=-X_{j}
\]
for $i\neq j$ is respected by $\rho$ since according to the proof
of Proposition~\ref{prop:properties of gen spin reps in simply laced case}
the condition $\left\{ A,B\right\} =0$ is equivalent to $\left[A,\left[A,B\right]\right]=-B$
as long as $A^{2}=B^{2}=-\frac{1}{4}\mathrm{id}_{s}$. 
\end{proof}

The existence of generalized spin representations has been established in \cite{Hainke}.

\begin{thm}
\label{thm:extensions of simply laced spin reps}For $1\leq r<n$
let $\mathfrak{k}_{\leq r}:=\langle X_{1},\dots,X_{r}\rangle$ denote
the subalgebra of $\mathfrak{k}$ that is generated by the first $r$
generators. Furthermore, let $\rho : \mathfrak{k}_{\leq r}\rightarrow\mathrm{End}\left(\mathbb{C}^{s}\right)$
be a generalized spin representation as in Definition~\ref{Def:gen spin rep for simply laced case}. 

If $X_{r+1}$ centralizes $\mathfrak{k}_{\leq r}$, that is to say
$X_{r+1}$ commutes with all generators $X_{1},\dots,X_{r}$, then
there exists a generalized spin representation $\rho' : \mathfrak{k}_{\leq r+1}\rightarrow\mathrm{End}\left(\mathbb{C}^{s}\right)$
with $\rho'_{\vert\mathfrak{k}_{\leq r}}=\rho$ given by sending $X_{r+1}$
to $\frac{1}{2}i\cdot \mathrm{id}_{s}$.

If $X_{r+1}$ does not centralize $\mathfrak{k}_{\leq r}$, then
$\rho$ can be extended to a generalized spin representation $\rho' : \mathfrak{k}_{\leq r+1}\rightarrow\mathrm{End}\left(\mathbb{C}^{s}\oplus\mathbb{C}^{s}\right)$.
For this define a sign automorphism $s_{0} : \mathfrak{k}_{\leq r}\rightarrow\mathfrak{k}_{\leq r}$
by
\[
s_{0}\left(X_{i}\right)=\begin{cases}
X_{i}, & \text{ if $(i,r+1)$ do not form an edge of the Dynkin diagram},\\
-X_{i}, & \text{ if $(i,r+1)$ form an edge of the Dynkin diagram},
\end{cases}
\]
and define the extension via
\[
\rho'_{\vert\mathfrak{k}_{\leq r}}=\rho\oplus\rho\circ s_{0}
\]
and 
\[
\rho'\left(X_{r+1}\right)=\frac{1}{2}\, \mathrm{id}_{s}\otimes\begin{pmatrix}0 & i\\
i & 0
\end{pmatrix}.
\]
\end{thm}
\begin{proof}
See \cite[Theorem~3.9]{Hainke}.
\end{proof}

\begin{cor}
  \label{cor:Construction of a gen spin rep. for simply laced case}
  Given
a simply laced Kac--Moody algebra $\mathfrak{g}(A)$ and a maximal coclique of size $r$, then there exists a generalized
spin representation $\rho\,:\ \mathfrak{k}\rightarrow\mathrm{End}\left(\mathbb{C}^{s}\right)$,
where $s=2^{n-r}$, with compact image.
\end{cor}
\begin{proof}
See \cite[Corollary~3.10 and Theorem~3.14]{Hainke}.
\end{proof}

\section{Extending a generalized
$\frac{1}{2}$-spin representation \\ --- following Kleinschmidt and Nicolai}    \label{sub:Extending-a-generalized}

Throughout this section let $\mathfrak{g}$ be a simply
laced split real Kac--Moody algebra with maximal compact subalgebra $\mathfrak{k}$.
By Corollary~\ref{cor:Construction of a gen spin rep. for simply laced case}
there exists a generalized $\frac{1}{2}$-spin representation $\rho : \mathfrak{k}\rightarrow\mathrm{End}(\mathbb{C}^{l})$.
In this section we make use of Clifford algebras in order to define higher generalized spin representations as carried out by Kleinschmidt and Nicolai \cite{AN1} for $E_{10}$.

Let $V\otimes S$ be the tensor product of two $\mathbb{R}$-vector spaces
$V$ with basis $\left\{ e^{1},\dots,e^{k}\right\}$ and $S$ with
basis $\left\{ f_{1},\dots,f_{l}\right\} $. Then $\left\{ e^{i}\otimes f_{j}\ \vert\ 1\leq i\leq k,\ 1\leq j\leq l\right\}$ is a natural $\mathbb{R}$-basis of $V \otimes S$.
Endow $V$ with a nondegenerate bilinear form
$q_{1}$ and $S$ with a positive definite bilinear form $q_{2}$ such
that the basis $\left\{ f_{1},\dots,f_{l}\right\}$ is orthonormal, i.e., 
\[
q_{2}\left(f_{\alpha},f_{\beta}\right)=\delta_{\alpha\beta}\ \quad \text{for $\alpha,\beta\in\{1,\dots,l\}$}.
\]
Let $\left(G^{ab}\right)_{1\leq a,b\leq k}$ denote the Gram matrix of $q_1$ with respect to the basis $\left\{ e^{1},\dots,e^{k}\right\}$, i.e., the matrix whose
components are given by 
\[
G^{ab}=q_{1}\left(e^{a},e^{b}\right)\ \quad \text{for $a,b\in\{1,\dots,k\}$}.
\]
Define $q:=q_1 \otimes q_2$ as the bilinear extension of $q_1$ and $q_2$ to $V\otimes S$ so that on the chosen basis
$\left\{ e^{i}\otimes f_{j}\ \vert\ 1\leq i\leq k,\ 1\leq j\leq l\right\}$ one has for $\alpha,\beta\in\{1,\dots,l\},\, a,b\in\{1,\dots,k\}$
\[
q\left(e^{a}\otimes f_{\alpha}, e^{b}\otimes f_{\beta}\right)=q_{1}\left(e^{a},e^{b}\right)\cdot q_{2}\left(f_{\alpha},f_{\beta}\right)=G^{ab}\delta_{\alpha\beta}.
\]
The bilinear form $q$ induces a quadratic form $$Q : V\otimes S \to \mathbb{R} : w \mapsto q(w,w).$$

One defines the Clifford algebra $\mathcal{S}=Cl(V\otimes S, Q)$ as the quotient of $\mathcal{T}\left(V\otimes S\right)$ modulo the ideal $I_{Q}$ generated by elements of the form
\[
w\otimes w-\frac{1}{2}Q(w) \cdot 1\ ,\ w\in V\otimes S.
\]
In $\mathcal{S}$ one therefore has $w^{2}=\frac{1}{2}Q(w)$, which via polarization
one can restate this as 
\begin{eqnarray}
wv+vw & = & q(v,w).    \label{differenceinnormalization}
\end{eqnarray}
On the level of the basis for $\alpha,\beta\in\{1,\dots,l\},\, a,b\in\{1,\dots,k\}$ this reads as
\[
\left(e^{a}\otimes f_{\alpha}\right)\left(e^{b}\otimes f_{\beta}\right)+\left(e^{b}\otimes f_{\beta}\right)\left(e^{a}\otimes f_{\alpha}\right)=G^{ab}\delta_{\alpha\beta},
\]
which one may repackage in a compact notation by defining for $\alpha\in\{1,\dots,l\}, \mathcal{A} \in\{1,\dots,k\}$  
\begin{eqnarray}
\phi_{\alpha}^{\mathcal{A}} & := & e^{\mathcal{A}}\otimes f_{\alpha}, \label{phidef}
\end{eqnarray}
thus yielding the identity for $\alpha,\beta\in\{1,\dots,l\},\,\mathcal{A},\mathcal{B}\in\{1,\dots,k\}$
\begin{equation}
\left\{ \phi_{\alpha}^{\mathcal{A}}\,,\,\phi_{\beta}^{\mathcal{B}}\right\} :=\phi_{\alpha}^{\mathcal{A}}\phi_{\beta}^{\mathcal{B}}+\phi_{\beta}^{\mathcal{B}}\phi_{\alpha}^{\mathcal{A}}=G^{\mathcal{AB}}\delta_{\alpha\beta}.\label{eq:anti-commutator in second quantization of root operators-1}
\end{equation}
Since $\left\{ e^{i}\otimes f_{j}\ \vert\ 1\leq i\leq k,\ 1\leq j\leq l\right\} $
is a basis of $V\otimes S$, the set $\left\{ \phi_{\alpha}^{\mathcal{A}}\ \vert\ 1\leq\mathcal{A}\leq k,\ 1\leq\alpha\leq l\right\} $
is a generating set of the $\mathbb{R}$-algebra $\mathcal{S}$. 

%
%

\begin{lem}
  \label{lem:commutator of bilinear ansatz}
For $X,Y \in \mathbb{R}^{k\times k}$ and $S,T \in \mathbb{R}^{l \times l}$ consider the following
elements of $\mathcal{S}$:
\begin{eqnarray*}
\hat{A} & :=& \sum_{\mathcal{A},\mathcal{B}=1}^k\sum_{\alpha,\beta=1}^lX_{\mathcal{AB}}S^{\alpha\beta}\phi_{\alpha}^{\mathcal{A}}\phi_{\beta}^{\mathcal{B}}, \\  \hat{B} & := & \sum_{\mathcal{C},\mathcal{D}=1}^k\sum_{\gamma,\delta=1}^lY_{\mathcal{CD}}T^{\gamma\delta}\phi_{\gamma}^{\mathcal{C}}\phi_{\delta}^{\mathcal{D}}.
\end{eqnarray*}
Under the hypothesis that for all $\alpha,\beta\in\{1,\dots,l\}$ and for all $\mathcal{A},\mathcal{B}\in\{1,\dots,k\}$
\begin{eqnarray}
  X_{\mathcal{AB}}S^{\alpha\beta}& = & -X_{\mathcal{BA}}S^{\beta\alpha} \notag \\
  Y_{\mathcal{AB}}T^{\alpha\beta}& = & -Y_{\mathcal{BA}}T^{\beta\alpha}\label{eq:antisymmetry-property of ansatz matrices-1} 
\end{eqnarray}
the commutator of $\hat{A}$
and $\hat{B}$ is equal to
\begin{equation}
\left[\hat{A},\hat{B}\right]=\sum_{\mathcal{A},\mathcal{B}=1}^k\sum_{\alpha,\beta=1}^l \phi_{\alpha}^{\mathcal{A}}\left(\left[X,Y\right]_{\mathcal{AB}}\left\{ S,T\right\} ^{\alpha\beta}+\left\{ X,Y\right\} _{\mathcal{AB}}\left[S,T\right]^{\alpha\beta}\right)\phi_{\beta}^{\mathcal{B}}\ ,\label{eq:commutator of bilinear ansatz-1}
\end{equation}
where
(anti-)commutators of $X$ and $Y$, resp.\ $S$ and $T$ are taken with respect
to the bilinear forms as follows:
\begin{eqnarray}
\left[X,Y\right]_{\mathcal{AB}} & = & \sum_{\mathcal{C},\mathcal{D}=1}^k\left(X_{\mathcal{AC}}G^{\mathcal{CD}}Y_{\mathcal{DB}}-Y_{\mathcal{AC}}G^{\mathcal{CD}}X_{\mathcal{DB}}\right), \label{so91commutator}\\
\left\{ X,Y\right\} _{\mathcal{AB}} & = & \sum_{\mathcal{C},\mathcal{D}=1}^k\left(X_{\mathcal{AC}}G^{\mathcal{CD}}Y_{\mathcal{DB}}+Y_{\mathcal{AC}}G^{\mathcal{CD}}X_{\mathcal{DB}}\right), \label{so91anticommutator} \\
\left[S,T\right]^{\alpha\beta} & = & \sum_{\gamma,\delta=1}^l\left(S^{\alpha\gamma}\delta_{\gamma\delta}T^{\delta\beta}-T^{\alpha\gamma}\delta_{\gamma\delta}S^{\delta\beta}\right), \label{genspincommutator} \\
\left\{ S,T\right\} ^{\alpha\beta} & = & \sum_{\gamma,\delta=1}^l\left(S^{\alpha\gamma}\delta_{\gamma\delta}T^{\delta\beta}+T^{\alpha\gamma}\delta_{\gamma\delta}S^{\delta\beta}\right). \label{genspinanticommutator}
\end{eqnarray}
\end{lem}

\begin{rem}
  \begin{enumerate}
    \item On level of the tensor product matrices, hypothesis \eqref{eq:antisymmetry-property of ansatz matrices-1} simply requires anti-symmetry: 
\[
X^{T}\otimes S^{T}=-X\otimes S.
\]
\item The statement of the lemma can be found as \cite[(4.17), p.~13 and footnote~10, p.~14]{AN1}. 
  \end{enumerate}
\end{rem}

\begin{proof}[Proof of Lemma~\ref{lem:commutator of bilinear ansatz}.]
One computes
\begin{eqnarray}
&& \left[X,Y\right]_{\mathcal{AB}}\left\{ S,T\right\}^{\alpha\beta} +\left\{ X,Y\right\}_{\mathcal{AB}} \left[S,T\right]^{\alpha\beta} \notag\\ & = & \left((XY)_{\mathcal{AB}}-(YX)_{\mathcal{AB}}\right)\left((ST)^{\alpha\beta}+(TS)^{\alpha\beta}\right)+\left((XY)_{\mathcal{AB}}+(YX)_{\mathcal{AB}}\right)\left((ST)^{\alpha\beta}-(TS)^{\alpha\beta}\right)\notag\\
 & = & \left(XY\right)_{\mathcal{AB}}\left(ST\right)^{\alpha\beta}+\left(XY\right)_{\mathcal{AB}}\left(TS\right)^{\alpha\beta}-\left(YX\right)_{\mathcal{AB}}\left(ST\right)^{\alpha\beta}-\left(YX\right)_{\mathcal{AB}}\left(TS\right)^{\alpha\beta}\notag\\
 &  & +\left(XY\right)_{\mathcal{AB}}\left(ST\right)^{\alpha\beta}-\left(XY\right)_{\mathcal{AB}}\left(TS\right)^{\alpha\beta}+\left(YX\right)_{\mathcal{AB}}\left(ST\right)^{\alpha\beta}-\left(YX\right)_{\mathcal{AB}}\left(TS\right)^{\alpha\beta}\notag\\
 & = & 2\left(XY\right)_{\mathcal{AB}}\left(ST\right)^{\alpha\beta}-2\left(YX\right)_{\mathcal{AB}}\left(TS\right)^{\alpha\beta}, \label{firstpreliminarycomputation}
\end{eqnarray}
where in analogy to the (anti-)commutators one abbreviates
\begin{eqnarray*}
  \left(XY\right)_{\mathcal{AB}}=\sum_{\mathcal{C},\mathcal{D}=1}^kX_{\mathcal{AC}}G^{\mathcal{CD}}Y_{\mathcal{DB}}, & &  (YX)_{\mathcal{AB}} =  \sum_{\mathcal{C},\mathcal{D}=1}^kY_{\mathcal{AC}}G^{\mathcal{CD}}X_{\mathcal{DB}}, \\
  (ST)^{\alpha\beta} = \sum_{\gamma,\delta=1}^lS^{\alpha\gamma}\delta_{\gamma\delta}T^{\delta\beta}, & & (TS)^{\alpha\beta} = \sum_{\gamma,\delta=1}^lT^{\alpha\gamma}\delta_{\gamma\delta}S^{\delta\beta}.
  \end{eqnarray*}
Several applications of equality~\eqref{eq:anti-commutator in second quantization of root operators-1}
yield
\begin{eqnarray*}
\left[\hat{A},\hat{B}\right] & = & \sum_{\mathcal{A},\mathcal{B},\mathcal{C},\mathcal{D}=1}^k \sum_{\alpha,\beta,\gamma,\delta=1}^lX_{\mathcal{AB}}S^{\alpha\beta}Y_{\mathcal{CD}}T^{\gamma\delta}\left(\phi_{\alpha}^{\mathcal{A}}\phi_{\beta}^{\mathcal{B}}\phi_{\gamma}^{\mathcal{C}}\phi_{\delta}^{\mathcal{D}}-\phi_{\gamma}^{\mathcal{C}}\phi_{\delta}^{\mathcal{D}}\phi_{\alpha}^{\mathcal{A}}\phi_{\beta}^{\mathcal{B}}\right)\\
 & = & \sum_{\mathcal{A},\mathcal{B},\mathcal{C},\mathcal{D}=1}^k \sum_{\alpha,\beta,\gamma,\delta=1}^lX_{\mathcal{AB}}S^{\alpha\beta}Y_{\mathcal{CD}}T^{\gamma\delta}\left(\phi_{\alpha}^{\mathcal{A}}\phi_{\beta}^{\mathcal{B}}\phi_{\gamma}^{\mathcal{C}}\phi_{\delta}^{\mathcal{D}}+\phi_{\gamma}^{\mathcal{C}}\phi_{\alpha}^{\mathcal{A}}\phi_{\delta}^{\mathcal{D}}\phi_{\beta}^{\mathcal{B}}-G^{\mathcal{DA}}\delta_{\delta\alpha}\phi_{\gamma}^{\mathcal{C}}\phi_{\beta}^{\mathcal{B}}\right)\\
 & = & \sum_{\mathcal{A},\mathcal{B},\mathcal{C},\mathcal{D}=1}^k \sum_{\alpha,\beta,\gamma,\delta=1}^lX_{\mathcal{AB}}S^{\alpha\beta}Y_{\mathcal{CD}}T^{\gamma\delta}\left(\phi_{\alpha}^{\mathcal{A}}\phi_{\beta}^{\mathcal{B}}\phi_{\gamma}^{\mathcal{C}}\phi_{\delta}^{\mathcal{D}}-\phi_{\alpha}^{\mathcal{A}}\phi_{\gamma}^{\mathcal{C}}\phi_{\delta}^{\mathcal{D}}\phi_{\beta}^{\mathcal{B}}+G^{\mathcal{CA}}\delta_{\gamma\alpha}\phi_{\delta}^{\mathcal{D}}\phi_{\beta}^{\mathcal{B}}-G^{\mathcal{DA}}\delta_{\delta\alpha}\phi_{\gamma}^{\mathcal{C}}\phi_{\beta}^{\mathcal{B}}\right)\\
 & = & \sum_{\mathcal{A},\mathcal{B},\mathcal{C},\mathcal{D}=1}^k \sum_{\alpha,\beta,\gamma,\delta=1}^lX_{\mathcal{AB}}S^{\alpha\beta}Y_{\mathcal{CD}}T^{\gamma\delta}\left(\phi_{\alpha}^{\mathcal{A}}\phi_{\beta}^{\mathcal{B}}\phi_{\gamma}^{\mathcal{C}}\phi_{\delta}^{\mathcal{D}}+\phi_{\alpha}^{\mathcal{A}}\phi_{\gamma}^{\mathcal{C}}\phi_{\beta}^{\mathcal{B}}\phi_{\delta}^{\mathcal{D}}-G^{\mathcal{DB}}\delta_{\delta\beta}\phi_{\alpha}^{\mathcal{A}}\phi_{\gamma}^{\mathcal{C}}\right)\\
 &  & +\sum_{\mathcal{A},\mathcal{B},\mathcal{C},\mathcal{D}=1}^k \sum_{\alpha,\beta,\gamma,\delta=1}^lX_{\mathcal{AB}}S^{\alpha\beta}Y_{\mathcal{CD}}T^{\gamma\delta}\left(G^{\mathcal{CA}}\delta_{\gamma\alpha}\phi_{\delta}^{\mathcal{D}}\phi_{\beta}^{\mathcal{B}}-G^{\mathcal{DA}}\delta_{\delta\alpha}\phi_{\gamma}^{\mathcal{C}}\phi_{\beta}^{\mathcal{B}}\right)\\
 & = & \sum_{\mathcal{A},\mathcal{B},\mathcal{C},\mathcal{D}=1}^k \sum_{\alpha,\beta,\gamma,\delta=1}^lX_{\mathcal{AB}}S^{\alpha\beta}Y_{\mathcal{CD}}T^{\gamma\delta}\left(\phi_{\alpha}^{\mathcal{A}}\phi_{\beta}^{\mathcal{B}}\phi_{\gamma}^{\mathcal{C}}\phi_{\delta}^{\mathcal{D}}-\phi_{\alpha}^{\mathcal{A}}\phi_{\beta}^{\mathcal{B}}\phi_{\gamma}^{\mathcal{C}}\phi_{\delta}^{\mathcal{D}}+G^{\mathcal{CB}}\delta_{\gamma\beta}\phi_{\alpha}^{\mathcal{A}}\phi_{\delta}^{\mathcal{D}}-G^{\mathcal{DB}}\delta_{\delta\beta}\phi_{\alpha}^{\mathcal{A}}\phi_{\gamma}^{\mathcal{C}}\right)\\
 &  & +\sum_{\mathcal{A},\mathcal{B},\mathcal{C},\mathcal{D}=1}^k \sum_{\alpha,\beta,\gamma,\delta=1}^lX_{\mathcal{AB}}S^{\alpha\beta}Y_{\mathcal{CD}}T^{\gamma\delta}\left(G^{\mathcal{CA}}\delta_{\gamma\alpha}\phi_{\delta}^{\mathcal{D}}\phi_{\beta}^{\mathcal{B}}-G^{\mathcal{DA}}\delta_{\delta\alpha}\phi_{\gamma}^{\mathcal{C}}\phi_{\beta}^{\mathcal{B}}\right)\\
 & = & \sum_{\mathcal{A},\mathcal{B},\mathcal{C},\mathcal{D}=1}^k \sum_{\alpha,\beta,\gamma,\delta=1}^lX_{\mathcal{AB}}S^{\alpha\beta}Y_{\mathcal{CD}}T^{\gamma\delta}\left(G^{\mathcal{CB}}\delta_{\gamma\beta}\phi_{\alpha}^{\mathcal{A}}\phi_{\delta}^{\mathcal{D}}-G^{\mathcal{DB}}\delta_{\delta\beta}\phi_{\alpha}^{\mathcal{A}}\phi_{\gamma}^{\mathcal{C}}+G^{\mathcal{CA}}\delta_{\gamma\alpha}\phi_{\delta}^{\mathcal{D}}\phi_{\beta}^{\mathcal{B}}-G^{\mathcal{DA}}\delta_{\delta\alpha}\phi_{\gamma}^{\mathcal{C}}\phi_{\beta}^{\mathcal{B}}\right).
\end{eqnarray*}
Using the symmetry of $G^{\mathcal{AB}}$ and of $\delta_{\alpha\beta}$ this is rearranged
to%
\begin{eqnarray*}
\left[\hat{A},\hat{B}\right] & = & \sum_{\mathcal{A},\mathcal{B},\mathcal{C},\mathcal{D}=1}^k\sum_{\alpha,\beta,\gamma,\delta=1}^l\phi_{\alpha}^{\mathcal{A}}X_{\mathcal{AB}}G^{\mathcal{BC}}Y_{\mathcal{CD}}S^{\alpha\beta}\delta_{\beta\gamma}T^{\gamma\delta}\phi_{\delta}^{\mathcal{D}}-\phi_{\alpha}^{\mathcal{A}}Y_{\mathcal{CD}}G^{\mathcal{DB}}X_{\mathcal{AB}}S^{\alpha\beta}T^{\gamma\delta}\delta_{\delta\beta}\phi_{\gamma}^{\mathcal{C}}\\
 &  & +\phi_{\delta}^{\mathcal{D}}Y_{\mathcal{CD}}G^{\mathcal{CA}}X_{\mathcal{AB}}S^{\alpha\beta}\delta_{\alpha\gamma}T^{\gamma\delta}\phi_{\beta}^{\mathcal{B}}-\phi_{\gamma}^{\mathcal{C}}Y_{\mathcal{CD}}G^{\mathcal{DA}}X_{\mathcal{AB}}T^{\gamma\delta}\delta_{\delta\alpha}S^{\alpha\beta}\phi_{\beta}^{\mathcal{B}}.
\end{eqnarray*}
This can then be transformed by renaming indices and using symmetry of the bilinear forms and anti-symmetry of the tensor product matrices (cf.\ \eqref{eq:antisymmetry-property of ansatz matrices-1}):
\begin{eqnarray*}
\left[\hat{A},\hat{B}\right] & = & \sum_{\mathcal{A},\mathcal{B},\mathcal{C},\mathcal{D}=1}^k\sum_{\alpha,\beta,\gamma,\delta=1}^l\phi_{\alpha}^{\mathcal{A}}\left(X_{\mathcal{AB}}G^{\mathcal{BC}}Y_{\mathcal{CD}}S^{\alpha\beta}\delta_{\beta\gamma}T^{\gamma\delta}-Y_{\mathcal{DC}}G^{\mathcal{CB}}X_{\mathcal{AB}}S^{\alpha\beta}T^{\delta\gamma}\delta_{\gamma\beta}\right)\phi_{\delta}^{\mathcal{D}}\\
 &  & +\phi_{\delta}^{\mathcal{D}}\left(Y_{\mathcal{CD}}G^{\mathcal{CA}}X_{\mathcal{AB}}S^{\alpha\beta}\delta_{\alpha\gamma}T^{\gamma\delta}-Y_{\mathcal{DC}}G^{\mathcal{CA}}X_{\mathcal{AB}}T^{\delta\gamma}\delta_{\gamma\alpha}S^{\alpha\beta}\right)\phi_{\beta}^{\mathcal{B}} \\
 & \stackrel{\eqref{eq:antisymmetry-property of ansatz matrices-1}}{=} & \sum_{\mathcal{A},\mathcal{B},\mathcal{C},\mathcal{D}=1}^k\sum_{\alpha,\beta,\gamma,\delta=1}^l\phi_{\alpha}^{\mathcal{A}}\left(X_{\mathcal{AB}}G^{\mathcal{BC}}Y_{\mathcal{CD}}S^{\alpha\beta}\delta_{\beta\gamma}T^{\gamma\delta}+X_{\mathcal{AB}}G^{\mathcal{BC}}Y_{\mathcal{CD}}S^{\alpha\beta}\delta_{\beta\gamma}T^{\gamma\delta}\right)\phi_{\delta}^{\mathcal{D}}\\
 &  & +\phi_{\delta}^{\mathcal{D}}\left(-Y_{\mathcal{DC}}G^{\mathcal{CA}}X_{\mathcal{AB}}T^{\delta\gamma}\delta_{\gamma\alpha}S^{\alpha\beta}-Y_{\mathcal{DC}}G^{\mathcal{CA}}X_{\mathcal{AB}}T^{\delta\gamma}\delta_{\gamma\alpha}S^{\alpha\beta}\right)\phi_{\beta}^{\mathcal{B}}\\
 & = & 2\sum_{\mathcal{A},\mathcal{D}=1}^k\sum_{\alpha,\delta=1}^l\phi_{\alpha}^{\mathcal{A}}\left(XY\right)_{\mathcal{AD}}\left(ST\right)^{\alpha\delta}\phi_{\delta}^{\mathcal{D}}-2\sum_{B,\mathcal{D}=1}^k\sum_{\beta,\delta=1}^l\phi_{\delta}^{\mathcal{D}}\left(YX\right)_{\mathcal{DB}}\left(TS\right)^{\delta\beta}\phi_{\beta}^{\mathcal{B}}\\
 & = & 2\sum_{\mathcal{A},\mathcal{D}=1}^k\sum_{\alpha,\delta=1}^l\phi_{\alpha}^{\mathcal{A}}\left[\left(XY\right)_{\mathcal{AD}}\left(ST\right)^{\alpha\delta}-\left(YX\right)_{\mathcal{AD}}\left(TS\right)^{\alpha\delta}\right]\phi_{\delta}^{\mathcal{D}}
\end{eqnarray*}
which in view of \eqref{firstpreliminarycomputation} completes the proof.
\end{proof}

\begin{rem}
In fact, we never used in the proof of Lemma~\ref{lem:commutator of bilinear ansatz} that the form $q_2$ be anisotropic. The computations hold in general for arbitrary non-degenerate forms. 
The definiteness of $q_2$ only becomes relevant now, when using the preceding lemma in order to construct various representations of $\mathfrak{k}$. The generalized spin representation $\rho : \mathfrak{k} \to \mathrm{End}(\mathbb{C}^s)$ from Corollary~\ref{cor:Construction of a gen spin rep. for simply laced case} provides anti-symmetric real $2s \times 2s$-matrices
\begin{eqnarray}
  \Gamma\left(\alpha_{i}\right):=2\rho\left(X_{i}\right) \label{remarkgamma}
\end{eqnarray}
for all simple roots $\alpha_{1},\dots,\alpha_{n}$ of $\mathfrak{g}$. Taking these as the matrix $S$ in the ansatz 
$$\hat{A} := \sum_{\mathcal{A},\mathcal{B}=1}^k\sum_{\alpha,\beta=1}^{l:=2s}X_{\mathcal{AB}}S^{\alpha\beta}\phi_{\alpha}^{\mathcal{A}}\phi_{\beta}^{\mathcal{B}}$$
of the lemma leaves one with the task of finding suitable symmetric matrices for $X$.

Note that, since we assumed $q_2$ to be anisotropic and conducted our computations with respect to an orthonormal basis for that form, the formulae given in \eqref{genspincommutator} and \eqref{genspinanticommutator} actually coincide with the standard definition of commutators and anti-commutators of matrices. In particular, the results from Proposition~\ref{prop:properties of gen spin reps in simply laced case} are applicable.
\end{rem}

\begin{defn} \label{defnlambda}
Now let $\lambda$ denote the finite set of real roots
\begin{equation}
\lambda:=\left\{ \alpha_{i} \mid 1\leq i\leq n\right\} \cup\left\{ \alpha_{i}+\alpha_{j} \in \Phi^{\mathrm{re}} \mid (i,j) \text{ form an edge of the Dynkin diagram}\right\}.\label{eq:set of sufficient roots Lambda}
\end{equation}
Note that for $\alpha,\beta\in\lambda$ one has $\left(\alpha\vert\beta\right)\in\{\pm1,0\}$.
\end{defn}

\begin{prop}
  \label{prop:criterion for root operator-1}
  A map $X : \lambda\rightarrow\mathbb{R}^{k\times k}$ that takes values in the set of symmetric matrices which satisfy for all $\alpha,\beta\in\lambda$
\begin{eqnarray}
\left[X(\alpha), X(\beta)\right] & = & 0,\quad\qquad\quad\quad\,\text{if }\left(\alpha\vert\beta\right)=0,\label{eq:commutator of X(alpha)-1}\\
\left\{ X(\alpha), X(\beta)\right\}  & = & \frac{1}{2}X(\alpha\pm\beta),\quad\text{if $\left(\alpha\vert\beta\right)=\mp1$ and $\alpha \pm \beta \in \lambda$},\label{eq:anticommutator of X(alpha)-1}
\end{eqnarray}
(with respect to the commutator and anti-commutator convention from \eqref{so91commutator} and \eqref{so91anticommutator}) 
 together with the anti-symmetric real matrices $\Gamma\left(\alpha_{1}\right),\dots,\Gamma\left(\alpha_{n}\right)$ from \eqref{remarkgamma}
turns the ansatz 
\[
\widehat{\mathrm{J}}(\alpha_{i})=\sum_{\mathcal{A},\mathcal{B}=1}^k\sum_{\alpha,\beta=1}^lX_{\mathcal{AB}}(\alpha_{i})\Gamma^{\alpha\beta}(\alpha_{i})\phi_{\alpha}^{\mathcal{A}}\phi_{\beta}^{\mathcal{B}}
\]
 into a finite-dimensional representation $\sigma$ of $\mathfrak{k}$ by defining $\sigma$ on the Berman generators $X_{1},\dots,X_{n}$ of $\mathfrak{k}$ as $\sigma\left(X_{i}\right):=\widehat{\mathrm{J}}(\alpha_{i})$. 
\end{prop}

\begin{rem}
The observation that \eqref{eq:commutator of X(alpha)-1} and \eqref{eq:anticommutator of X(alpha)-1} are the key identities for extending generalized spin representations has been made in \cite[(4.23), p.~15; (5.1), p.~18]{AN1}.
\end{rem}

\begin{proof}[Proof of Proposition~\ref{prop:criterion for root operator-1}.]
By the homomorphism theorem it suffices to establish that the commutator $\left[\widehat{\text{J}}(\alpha_{i}),\widehat{\text{J}}(\alpha_{j})\right]$ satisfies the relations from Theorem~\ref{gensandrels}.
By Lemma~\ref{lem:commutator of bilinear ansatz} one has
\begin{eqnarray*}
\left[\widehat{\text{J}}(\alpha_{i}),\widehat{\text{J}}(\alpha_{j})\right] & = & \sum_{\mathcal{A},\mathcal{B}=1}^k\sum_{\alpha,\beta=1}^l\phi_{\alpha}^{\mathcal{A}}\left[X\left(\alpha_{i}\right),X\left(\alpha_{j}\right)\right]_{\mathcal{AB}}\left\{ \Gamma\left(\alpha_{i}\right),\Gamma\left(\alpha_{j}\right)\right\} ^{\alpha\beta}\phi_{\beta}^{\mathcal{B}}\\
 &  & +\sum_{\mathcal{A},\mathcal{B}=1}^k\sum_{\alpha,\beta=1}^l\phi_{\alpha}^{\mathcal{A}}\left\{ X\left(\alpha_{i}\right),X\left(\alpha_{j}\right)\right\} _{\mathcal{AB}}\left[\Gamma\left(\alpha_{i}\right),\Gamma\left(\alpha_{j}\right)\right]^{\alpha\beta}\phi_{\beta}^{\mathcal{B}}\ .
\end{eqnarray*}
In case $(i,j)$ is not an edge of the Dynkin diagram this yields 
\[
\left[\widehat{\text{J}}(\alpha_{i}),\widehat{\text{J}}(\alpha_{j})\right]=0
\]
as desired, because in this case one has $\left[\Gamma\left(\alpha_{i}\right),\Gamma\left(\alpha_{j}\right)\right]=0$
by Proposition~\ref{prop:properties of gen spin reps in simply laced case} and, furthermore, $\left(\alpha_{i}\vert\alpha_{j}\right)=0$, i.e., $\left[X\left(\alpha_{i}\right),X\left(\alpha_{j}\right)\right]=0$
by hypothesis \eqref{eq:commutator of X(alpha)-1}.

In case $(i,j)$ is an edge of the Dynkin diagram one has $\left\{ \Gamma\left(\alpha_{i}\right),\Gamma\left(\alpha_{j}\right)\right\} =0$ by Proposition~\ref{prop:properties of gen spin reps in simply laced case}
and so
\[
\left\{ X\left(\alpha_{i}\right),X\left(\alpha_{j}\right)\right\} =\frac{1}{2}X\left(\alpha_{i}+\alpha_{j}\right)
\]
by hypothesis \eqref{eq:anticommutator of X(alpha)-1}.
Thus,
\[
\left[\widehat{\text{J}}(\alpha_{i}),\widehat{\text{J}}(\alpha_{j})\right]=\sum_{\mathcal{A},\mathcal{B}=1}^k\sum_{\alpha,\beta=1}^l\phi_{\alpha}^{\mathcal{A}}\cdot\frac{1}{2}X\left(\alpha_{i}+\alpha_{j}\right)_{\mathcal{AB}}\left[\Gamma\left(\alpha_{i}\right),\Gamma\left(\alpha_{j}\right)\right]^{\alpha\beta}\phi_{\beta}^{\mathcal{B}}\ .
\]
Applying the commutator with $\widehat{\text{J}}(\alpha_{i})$ again according to Lemma~\ref{lem:commutator of bilinear ansatz}
yields
\begin{eqnarray*}
\left[\widehat{\text{J}}(\alpha_{i}),\left[\widehat{\text{J}}(\alpha_{i}),\widehat{\text{J}}(\alpha_{j})\right]\right] & = & \frac{1}{2}\sum_{\mathcal{A},\mathcal{B}=1}^k\sum_{\alpha,\beta=1}^l\phi_{\alpha}^{\mathcal{A}}\left[X\left(\alpha_{i}\right), X\left(\alpha_{i}+\alpha_{j}\right)\right]_{\mathcal{AB}}\left\{ \Gamma\left(\alpha_{i}\right),\left[\Gamma\left(\alpha_{i}\right),\Gamma\left(\alpha_{j}\right)\right]\right\} ^{\alpha\beta}\phi_{\beta}^{\mathcal{B}}\\
 &  & +\frac{1}{2}\sum_{\mathcal{A},\mathcal{B}=1}^k\sum_{\alpha,\beta=1}^l\phi_{\alpha}^{\mathcal{A}}\left\{ X\left(\alpha_{i}\right),X\left(\alpha_{i}+\alpha_{j}\right)\right\} _{\mathcal{AB}}\left[\Gamma\left(\alpha_{i}\right),\left[\Gamma\left(\alpha_{i}\right),\Gamma\left(\alpha_{j}\right)\right]\right]^{\alpha\beta}\phi_{\beta}^{\mathcal{B}}\ .
\end{eqnarray*}
Since $\left(\alpha_{i}\vert\alpha_{i}+\alpha_{j}\right)=1$, by hypothesis \eqref{eq:anticommutator of X(alpha)-1} one has
\[
\left\{ X\left(\alpha_{i}\right),X\left(\alpha_{i}+\alpha_{j}\right)\right\} =\frac{1}{2}X\left(\alpha_{j}\right).
\]
Moreover,
\[
\left[\Gamma\left(\alpha_{i}\right),\left[\Gamma\left(\alpha_{i}\right),\Gamma\left(\alpha_{j}\right)\right]\right]=-4\Gamma\left(\alpha_{j}\right)
\]
because $\rho(X_{i})=\frac{1}{2}\Gamma(\alpha_{i})$ is a generalized
spin representation of $\mathfrak{k}$ (cf.\ Proposition~\ref{prop:properties of gen spin reps in simply laced case} and its proof). Furthermore,
\begin{eqnarray*}
\left\{ \Gamma\left(\alpha_{i}\right),\left[\Gamma\left(\alpha_{i}\right),\Gamma\left(\alpha_{j}\right)\right]\right\} & = & \Gamma\left(\alpha_{i}\right)\Gamma\left(\alpha_{i}\right)\Gamma\left(\alpha_{j}\right)-\Gamma\left(\alpha_{i}\right)\Gamma\left(\alpha_{j}\right)\Gamma\left(\alpha_{i}\right)\\
 &  & +\Gamma\left(\alpha_{i}\right)\Gamma\left(\alpha_{j}\right)\Gamma\left(\alpha_{i}\right)-\Gamma\left(\alpha_{j}\right)\Gamma\left(\alpha_{i}\right)\Gamma\left(\alpha_{i}\right)\\
 & = & 0,
\end{eqnarray*}
because $\Gamma\left(\alpha_{i}\right)\Gamma\left(\alpha_{i}\right)$ commutes with $\Gamma\left(\alpha_{j}\right)$ (cf.\ Corollary~\ref{cor:Sufficient conditions for gen spin rep}).
Altogether,
\begin{eqnarray*}
\left[\widehat{\text{J}}(\alpha_{i}),\left[\widehat{\text{J}}(\alpha_{i}),\widehat{\text{J}}(\alpha_{j})\right]\right] & = & \frac{1}{2}\sum_{\mathcal{A},\mathcal{B}=1}^k\sum_{\alpha,\beta=1}^l\phi_{\alpha}^{\mathcal{A}}\frac{1}{2}X\left(\alpha_{j}\right)_{\mathcal{AB}}\left(-4\Gamma\left(\alpha_{j}\right)\right)^{\alpha\beta}\phi_{\beta}^{\mathcal{B}}\\
 & = & -\sum_{\mathcal{A},\mathcal{B}}\phi_{\alpha}^{\mathcal{A}}X\left(\alpha_{j}\right)_{\mathcal{AB}}\Gamma\left(\alpha_{j}\right)^{\alpha\beta}\phi_{\beta}^{\mathcal{B}}\\
 & = & -\widehat{\text{J}}(\alpha_{j}),
\end{eqnarray*}
again as desired in view of Theorem~\ref{gensandrels}.

One concludes that the assignment
\[
\sigma\left(X_{i}\right):=\widehat{\text{J}}(\alpha_{i})
\]
defines a finite-dimensional representation of $\mathfrak{k}$.
\end{proof}
%

\section{Extending a generalized
$\frac{1}{2}$-spin representation \\ --- a coordinate-free approach}    \label{sub:Extending-a-generalized2}

In this section we discuss a coordinate-free version of Proposition~\ref{prop:criterion for root operator-1}.
We stress that in this section we make use of the usual definition of (anti-)commutators: For endomorphisms of a real vector space  $V$ define the (anti-)commutator as
\begin{eqnarray*}
\left[\cdot,\cdot\right] :  \mathrm{End}(V)\times\mathrm{End}(V) & \rightarrow & \mathrm{End}(V)\\
\left[A,B\right] & \mapsto & A\circ B-B\circ A
\end{eqnarray*}
and 
\begin{eqnarray*}
\left\{ \cdot,\cdot\right\} : \mathrm{End}(V) \times\mathrm{End}(V) & \rightarrow & \mathrm{End}(V)\\
\left\{ A, B\right\}  & \mapsto & A\circ B+B\circ A
\end{eqnarray*}
where $A,B\in\mathrm{End}(V)$ and $\circ$ denotes concatenation of
(linear) maps.

As in Definition~\ref{defnlambda} let
\[
\lambda:=\left\{ \alpha_{i} \mid 1\leq i\leq n\right\} \cup\left\{ \alpha_{i}+\alpha_{j} \in \Phi^{\mathrm{re}} \mid (i,j) \text{ form an edge of the Dynkin diagram}\right\} .
\]

\begin{prop} \label{coordinatefree}
A map $X : \lambda\rightarrow\mathrm{End}\left(V\right)$ satisfying for all $\alpha,\beta\in\lambda$
\begin{eqnarray}
\left[X(\alpha), X(\beta)\right] & = & 0\quad\qquad\quad\quad\,\text{if }\left(\alpha\vert\beta\right)=0\label{eq:commutator for X as endo}\\
\left\{ X(\alpha),X(\beta)\right\}  & = & X(\alpha\pm\beta)\quad\text{if $\left(\alpha\vert\beta\right)=\mp1$ and $\alpha \pm \beta \in \lambda$} \label{eq:anticommutator for X as endo}
\end{eqnarray}
provides
a finite-dimensional representation $\sigma$ of $\mathfrak{k}$
via the assignment $$\sigma\left(X_{i}\right):=X\left(\alpha_{i}\right)\otimes\Gamma\left(\alpha_{i}\right)\in\mathrm{End}\left(V\otimes S\right)$$ on the Berman generators $X_{1},\dots,X_{n}$ of $\mathfrak{k}$, where the $\Gamma(\alpha_i)$, $1 \leq i \leq n$ are the anti-symmetric real matrices from \eqref{remarkgamma}.
\end{prop}

\begin{rem} \label{Paul}
  \begin{enumerate}
\item Defining $X : \lambda\rightarrow\mathrm{End}\left(\mathbb{R}\right) = \mathbb{R}$ as the constant map $X \equiv \frac{1}{2}$ provides the generalized spin representation from \cite{Hainke}, cf.\ Corollary~\ref{cor:Sufficient conditions for gen spin rep}. On the other hand, in the approach taken by Kleinschmidt and Nicolai described in Proposition~\ref{prop:criterion for root operator-1} one needs to define $X : \lambda\rightarrow\mathbb{R}$ as the constant map $X \equiv \frac{1}{4}$ in order to obtain the generalized spin representation from \cite{Hainke}. This difference in normalization stems from the differences in normalizations of the underlying Clifford algebras when comparing \cite[Example~3.2]{Hainke} with \eqref{differenceinnormalization} on page \pageref{differenceinnormalization}. Similar differences are visible in the formulae for the $\frac{3}{2}$-spin representations given in \eqref{3/2-identity} on page \pageref{3/2-identity} and \eqref{formula32} on page \pageref{formula32} below.
\item Contrary to Proposition~\ref{prop:criterion for root operator-1}, the above coordinate-free version does not require the map $X : \lambda \to \mathrm{End}(V)$ to take images in the set of self-adjoint/symmetric operators.
\item Paul Levy pointed out to us the following. Let $W$ be the Weyl group of $\mathfrak{g}$ and let $\rho : W \to \mathrm{GL}(V)$ be a representation. The ansatz $X(\alpha) := \rho(s_\alpha) - \mathrm{id}$ leads to
  \begin{eqnarray*}
    && \rho(s_\alpha s_\beta) + \rho(s_\beta s_\alpha) - \rho(s_\alpha) - \rho(s_\beta) + \frac{1}{2} \mathrm{id} \\
    & = & \left(\rho(s_\alpha) - \frac{1}{2} \mathrm{id}\right)\left(\rho(s_\beta) - \frac{1}{2} \mathrm{id}\right) + \left(\rho(s_\beta) - \frac{1}{2} \mathrm{id}\right)\left(\rho(s_\alpha) - \frac{1}{2} \mathrm{id}\right) \\
    & = & \left\{ X(\alpha), X(\beta) \right\} \\
    & = & X(\alpha+\beta) \\
    & = & \rho(s_{\alpha+\beta}) - \frac{1}{2} \mathrm{id} \\
    & = & \rho(s_\alpha s_\beta s_\alpha) - \frac{1}{2} \mathrm{id}
  \end{eqnarray*}  
  for each pair $\alpha$, $\beta$ forming an $A_2$-subdiagram. One concludes that $$\left\{ X(\alpha), X(\beta) \right\} =  X(\alpha+\beta)$$ in fact is equivalent to
\begin{eqnarray}
  \rho(s_\alpha s_\beta s_\alpha) - \rho(s_\alpha s_\beta) - \rho(s_\beta s_\alpha) + \rho(s_\alpha) + \rho(s_\beta) - \mathrm{id} = 0. \label{weylidentity}
  \end{eqnarray}
Similar computations imply that in fact any case covered by \eqref{eq:anticommutator for X as endo} using the ansatz $X(\alpha) := \rho(s_\alpha) - \mathrm{id}$ is equivalent to \eqref{weylidentity}. Furthermore, one quickly computes that $[\rho(s_\alpha) - \mathrm{id},\rho(s_\beta) - \mathrm{id}] = 0$ whenever $(\alpha|\beta)=0$, because this is equivalent to $s_\alpha s_\beta = s_\beta s_\alpha$. We conclude that for the ansatz $X(\alpha) := \rho(s_\alpha) - \mathrm{id}$ it suffices to check \eqref{weylidentity} for each pair $\alpha$, $\beta$ forming an $A_2$-subdiagram.
\item Paul Levy also pointed out to us that the identity $$\rho(s_\alpha s_\beta s_\alpha) - \rho(s_\alpha s_\beta) - \rho(s_\beta s_\alpha) + \rho(s_\alpha) + \rho(s_\beta) - \mathrm{id} = 0$$ holds if and only if the given representation $W \geq \mathrm{Sym}_3 = \langle s_\alpha, s_\beta \rangle \to \mathrm{GL}(V) : w \mapsto \rho(w)$ does not contain a sign representation as an irreducible component. Indeed, among the irreducible representations of $\mathrm{Sym}_3$ the trivial and the geometric representations satisfy $\eqref{weylidentity}$ whereas the sign representation does not. 
\end{enumerate}
  \end{rem}

\begin{proof}[Proof of Proposition~\ref{coordinatefree}.]
By the homomorphism theorem it suffices to establish that the commutator $\left[\sigma\left(X_{i}\right),\sigma\left(X_{j}\right)\right]$ satisfies the relations from Theorem~\ref{gensandrels}.

  In case $(i,j)$ do not form an edge, one computes the following: 
\begin{eqnarray*}
\left[\sigma\left(X_{i}\right),\sigma\left(X_{j}\right)\right] & = & \left(X\left(\alpha_{i}\right)\otimes\Gamma\left(\alpha_{i}\right)\right)\circ\left(X\left(\alpha_{j}\right)\otimes\Gamma\left(\alpha_{j}\right)\right)\\
 &  & -\left(X\left(\alpha_{j}\right)\otimes\Gamma\left(\alpha_{j}\right)\right)\circ\left(X\left(\alpha_{i}\right)\otimes\Gamma\left(\alpha_{i}\right)\right)\\
 & = & X\left(\alpha_{i}\right)X\left(\alpha_{j}\right)\otimes\Gamma\left(\alpha_{i}\right)\Gamma\left(\alpha_{j}\right)-X\left(\alpha_{j}\right)X\left(\alpha_{i}\right)\otimes\Gamma\left(\alpha_{j}\right)\Gamma\left(\alpha_{i}\right)\\
 & = & X\left(\alpha_{i}\right)X\left(\alpha_{j}\right)\otimes\Gamma\left(\alpha_{i}\right)\Gamma\left(\alpha_{j}\right)-X\left(\alpha_{j}\right)X\left(\alpha_{i}\right)\otimes\Gamma\left(\alpha_{i}\right)\Gamma\left(\alpha_{j}\right)\\
 &  & +X\left(\alpha_{j}\right)X\left(\alpha_{i}\right)\otimes\Gamma\left(\alpha_{i}\right)\Gamma\left(\alpha_{j}\right)-X\left(\alpha_{j}\right)X\left(\alpha_{i}\right)\otimes\Gamma\left(\alpha_{j}\right)\Gamma\left(\alpha_{i}\right)\\
 & = & \left[X\left(\alpha_{i}\right),X\left(\alpha_{j}\right)\right]\otimes\Gamma\left(\alpha_{i}\right)\Gamma\left(\alpha_{j}\right)+X\left(\alpha_{j}\right)X\left(\alpha_{i}\right)\otimes\left[\Gamma\left(\alpha_{i}\right),\Gamma\left(\alpha_{j}\right)\right]\\
 & = & 0,
\end{eqnarray*}
because $\left[X(\alpha_i), X(\alpha_j)\right] = 0$ by hypothesis \eqref{eq:commutator for X as endo} and $\left[\Gamma\left(\alpha_{i}\right),\left(\alpha_{j}\right)\right]=0$ by Proposition~\ref{prop:properties of gen spin reps in simply laced case}.

In case $(i,j)$ is an edge, Proposition~\ref{prop:properties of gen spin reps in simply laced case} and hypothesis \eqref{eq:anticommutator for X as endo} yield
\[
\left\{ \Gamma\left(\alpha_{i}\right),\Gamma\left(\alpha_{j}\right)\right\} =0 \quad\quad \text{and} \quad\quad
\left\{ X\left(\alpha_{i}\right),X\left(\alpha_{j}\right)\right\} =X\left(\alpha_{i}+\alpha_{j}\right).
\]
Hence
\begin{eqnarray*}
\left[\sigma\left(X_{i}\right),\sigma\left(X_{j}\right)\right] & = & X\left(\alpha_{i}\right)X\left(\alpha_{j}\right)\otimes\Gamma\left(\alpha_{i}\right)\Gamma\left(\alpha_{j}\right)-X\left(\alpha_{j}\right)X\left(\alpha_{i}\right)\otimes\Gamma\left(\alpha_{j}\right)\Gamma\left(\alpha_{i}\right)\\
 & = & X\left(\alpha_{i}\right)X\left(\alpha_{j}\right)\otimes\Gamma\left(\alpha_{i}\right)\Gamma\left(\alpha_{j}\right)+X\left(\alpha_{j}\right)X\left(\alpha_{i}\right)\otimes\Gamma\left(\alpha_{i}\right)\Gamma\left(\alpha_{j}\right)\\
 &  & -X\left(\alpha_{j}\right)X\left(\alpha_{i}\right)\otimes\Gamma\left(\alpha_{i}\right)\Gamma\left(\alpha_{j}\right)-X\left(\alpha_{j}\right)X\left(\alpha_{i}\right)\otimes\Gamma\left(\alpha_{j}\right)\Gamma\left(\alpha_{i}\right)\\
 & = & \left\{ X\left(\alpha_{i}\right),X\left(\alpha_{j}\right)\right\} \otimes\Gamma\left(\alpha_{i}\right)\Gamma\left(\alpha_{j}\right)-X\left(\alpha_{j}\right)X\left(\alpha_{i}\right)\otimes\left\{ \Gamma\left(\alpha_{i}\right),\Gamma\left(\alpha_{j}\right)\right\} \\
 & = & X\left(\alpha_{i}+\alpha_{j}\right)\otimes\Gamma\left(\alpha_{i}\right)\Gamma\left(\alpha_{j}\right).
\end{eqnarray*}
Moreover, since the matrices $\frac{1}{2}\Gamma(\alpha_1)$, ..., $\frac{1}{2}\Gamma(\alpha_n)$ provide a generalized spin representation, by definition one has $\Gamma\left(\alpha_{i}\right)^{2}=4\rho\left(X_{i}\right)^{2}=-\mathrm{id}_{S}$
and by Proposition~\ref{prop:properties of gen spin reps in simply laced case} the matrices $\Gamma\left(\alpha_{i}\right)$ and $\Gamma\left(\alpha_{j}\right)$
anti-commute. Therefore one has the following:
\begin{eqnarray*}
\left[\sigma\left(X_{i}\right),\left[\sigma\left(X_{i}\right),\sigma\left(X_{j}\right)\right]\right] & = & \left(X\left(\alpha_{i}\right)\otimes\Gamma\left(\alpha_{i}\right)\right)\circ\left(X\left(\alpha_{i}+\alpha_{j}\right)\otimes\Gamma\left(\alpha_{i}\right)\Gamma\left(\alpha_{j}\right)\right)\\
 &  & -\left(X\left(\alpha_{i}+\alpha_{j}\right)\otimes\Gamma\left(\alpha_{i}\right)\Gamma\left(\alpha_{j}\right)\right)\circ\left(X\left(\alpha_{i}\right)\otimes\Gamma\left(\alpha_{i}\right)\right)\\
 & = & \left(X\left(\alpha_{i}\right)X\left(\alpha_{i}+\alpha_{j}\right)\right)\otimes\left(\Gamma\left(\alpha_{i}\right)\Gamma\left(\alpha_{i}\right)\Gamma\left(\alpha_{j}\right)\right)\\
 &  & -\left(X\left(\alpha_{i}+\alpha_{j}\right)X\left(\alpha_{i}\right)\right)\otimes\left(\Gamma\left(\alpha_{i}\right)\Gamma\left(\alpha_{j}\right)\Gamma\left(\alpha_{i}\right)\right)\\
 & = & -\left(X\left(\alpha_{i}\right)X\left(\alpha_{i}+\alpha_{j}\right)\right)\otimes\Gamma\left(\alpha_{j}\right)\\
 &  & -\left(X\left(\alpha_{i}+\alpha_{j}\right)X\left(\alpha_{i}\right)\right)\otimes\Gamma\left(\alpha_{j}\right)\\
 & = & -\left\{ X\left(\alpha_{i}\right),X\left(\alpha_{i}+\alpha_{j}\right)\right\} \otimes\Gamma\left(\alpha_{j}\right)\\
 & = & -X\left(\alpha_{j}\right)\otimes\Gamma\left(\alpha_{j}\right)=-\sigma\left(X_{j}\right). \qedhere
\end{eqnarray*}
\end{proof}

\section{Towards $\frac{3}{2}$-spin representations} \label{3/2}

Let $V:=\mathfrak{h}^{\ast}$. If the generalized Cartan matrix
$A$ is invertible, then $V=\text{span}_{\mathbb{R}}\{\alpha_{1},\dots,\alpha_{n}\}$; otherwise $V$ is of higher dimension $k := 2n-\mathrm{rk}(A)$. In both cases the invariant bilinear form on $\mathfrak{g}$ induces a nondegenerate bilinear form $\left(\cdot\vert\cdot\right)$ on $V$. Let $v^1$, \dots, $v^k$ be a basis of $V$ and define
\[
G^{ab}:=\left(v^{a},v^{b}\right).
\]
That is, $(G^{ab})_{1 \leq a, b \leq k}$ is the Gram matrix of the bilinear form $\left(\cdot\vert\cdot\right)$ on $V$ with respect to the basis $v^1$,\dots, $v^k$. Moreover, define $(G_{ab})_{1 \leq a, b \leq k} := {(G^{ab})_{1 \leq a, b \leq k}}^{-1}$, i.e.,
\[
\sum_{b=1}^kG^{ab}G_{bc}=\delta_{ac}.
\]
Note that Cramer's rule implies that also the matrix $(G_{ab})_{1 \leq a, b \leq k}$ is symmetric.

\begin{prop}\label{prop:the spin 3/2 extension}
  Let $\alpha=\sum_{i=1}^k\alpha_iv^i \in V = \mathfrak{h}^*$ be a real root. Then the map
\begin{eqnarray*}
  X : \lambda & \rightarrow & \mathbb{R}^{k\times k} \\ \alpha & \mapsto & (X(\alpha)_{ab})_{1 \leq a,b \leq k}
  \end{eqnarray*}
defined via
\begin{eqnarray}
X(\alpha)_{ab} & = & -\frac{1}{2}\alpha_{a}\alpha_{b}+\frac{1}{4}G_{ab}.  \label{3/2-identity}
\end{eqnarray}
yields a set of matrices that satisfy hypotheses \eqref{eq:commutator of X(alpha)-1}
and \eqref{eq:anticommutator of X(alpha)-1} of Proposition~\ref{prop:criterion for root operator-1}. In particular, this provides a finite-dimensional representation of $\mathfrak{k}$ via 
$\sigma(X_{i})= X\left(\alpha_{i}\right)\otimes\Gamma\left(\alpha_{i}\right)$.
\end{prop}

\begin{rem}
Formula \eqref{3/2-identity} is \cite[(4.21), p.~15]{AN1}.
\end{rem}

\begin{proof}[Proof of Proposition~\ref{prop:the spin 3/2 extension}]
It suffices to establish the hypotheses of Proposition~\ref{prop:criterion for root operator-1}. Note first that the matrices $X(\alpha)$ are symmetric by definition.
For $\alpha, \beta \in \lambda$ with 
$\left(\alpha\vert\beta\right)=0$ one computes
\begin{eqnarray*}
\left[X(\alpha),X(\beta)\right]_{ad} & \stackrel{\eqref{so91commutator}}{=} & \sum_{b,c=1}^k\left(-\frac{1}{2}\alpha_{a}\alpha_{b}+\frac{1}{4}G_{ab}\right)G^{bc}\left(-\frac{1}{2}\beta_{c}\beta_{d}+\frac{1}{4}G_{cd}\right)\\
 &  & -\sum_{b,c=1}^k\left(-\frac{1}{2}\beta_{a}\beta_{b}+\frac{1}{4}G_{ab}\right)G^{bc}\left(-\frac{1}{2}\alpha_{c}\alpha_{d}+\frac{1}{4}G_{cd}\right)\\
 & = & \sum_{b,c=1}^k \left(\frac{1}{4}\alpha_{a}\alpha_{b}G^{bc}\beta_{c}\beta_{d}-\frac{1}{4}\beta_{a}\beta_{b}G^{bc}\alpha_{c}\alpha_{d}-\frac{1}{8}\alpha_{a}\alpha_{b}G^{bc}G_{cd} \right.\\
 &  & + \frac{1}{8}\beta_{a}\beta_{b}G^{bc}G_{cd}-\frac{1}{8}G_{ab}G^{bc}\beta_{c}\beta_{d}+\frac{1}{8}G_{ab}G^{bc}\alpha_{c}\alpha_{d}\\
 &  & + \left. \frac{1}{16}G_{ab}G^{bc}G_{cd}-\frac{1}{16}G_{ab}G^{bc}G_{cd} \right) \\
 & = & \frac{1}{4}\alpha_{a}\left(\alpha\vert\beta\right)\beta_{d}-\frac{1}{4}\beta_{a}\left(\alpha\vert\beta\right)\alpha_{d}-\frac{1}{8}\alpha_{a}\alpha_{d}+\frac{1}{8}\beta_{a}\beta_{d}  \\
 &  & -\frac{1}{8}\beta_{a}\beta_{d}+\frac{1}{8}\alpha_{a}\alpha_{d} \\
 & = & 0.
\end{eqnarray*}
Moreover, for $\left(\alpha\vert\beta\right)=\mp1$ one computes
\begin{eqnarray*}
\left\{ X(\alpha),X(\beta)\right\} _{ad} & \stackrel{\eqref{so91anticommutator}}{=} & \sum_{b,c=1}^k \left(-\frac{1}{2}\alpha_{a}\alpha_{b}+\frac{1}{4}G_{ab}\right)G^{bc}\left(-\frac{1}{2}\beta_{c}\beta_{d}+\frac{1}{4}G_{cd}\right)\\
 &  & + \sum_{b,c=1}^k\left(-\frac{1}{2}\beta_{a}\beta_{b}+\frac{1}{4}G_{ab}\right)G^{bc}\left(-\frac{1}{2}\alpha_{c}\alpha_{d}+\frac{1}{4}G_{cd}\right)\\
 & = &  \frac{1}{4}\alpha_{a}\beta_{d}\left(\alpha\vert\beta\right)-\frac{1}{8}\alpha_{a}\alpha_{d}-\frac{1}{8}\beta_{a}\beta_{d}+\frac{1}{16}G_{ab}G^{bc}G_{cd} \\
 &  & +\frac{1}{4}\beta_{a}\alpha_{d}\left(\alpha\vert\beta\right)-\frac{1}{8}\beta_{a}\beta_{d}-\frac{1}{8}\alpha_{a}\alpha_{d}+\frac{1}{16}G_{ab}G^{bc}G_{cd} \\
 & = & \frac{1}{4}\left(-\alpha_{a}\alpha_{d}-\beta_{a}\beta_{d}\mp\left(\alpha_{a}\beta_{d}+\beta_{a}\alpha_{d}\right)+\frac{1}{2}G_{ad}\right)\\
 & = & \frac{1}{2}\left(-\frac{1}{2}\left(\alpha_{a}\pm\beta_{a}\right)\left(\alpha_{d}\pm\beta_{d}\right)+\frac{1}{4}G_{ad}\right)\\
 & = & \frac{1}{2}X(\alpha\pm\beta)_{ad}.     \qedhere
\end{eqnarray*}
\end{proof}

We conclude this section with the following coordinate-free version of Proposition~\ref{prop:the spin 3/2 extension}.

\begin{prop} \label{32coordinatefree}
For $V=\mathfrak{h}^{\ast}$ let $\left(\cdot\vert\cdot\right)$ denote
the induced invariant bilinear form on $\mathfrak{h}^{\ast}$. Define
$X : \Delta^{\mathrm{re}}\rightarrow\mathrm{End}\left(\mathfrak{h}^{\ast}\right)$
via 
\begin{eqnarray}
\alpha\mapsto X(\alpha) & := & -\alpha\left(\alpha\vert\cdot\right)+\frac{1}{2}id_{\mathfrak{h}^{\ast}}.    \label{formula32}
\end{eqnarray}
Then $X$ satisfies \eqref{eq:commutator for X as endo}  and
\eqref{eq:anticommutator for X as endo} for all real roots $\alpha,\beta$
with $\left(\alpha\vert\beta\right)\in\left\{ 0,\pm1\right\} $ and
thus provides a representation $\sigma$ of $\mathfrak{k}$.
\end{prop}

\begin{proof}
First consider $\alpha,\beta\in\Delta^{\mathrm{re}}$ such that $\left(\alpha\vert\beta\right)=0$.
Then one has 
\begin{eqnarray*}
\left[X(\alpha),X(\beta)\right] & = & \left(-\alpha\left(\alpha\vert\cdot\right)+\frac{1}{2}id_{\mathfrak{h}^{\ast}}\right)\left(-\beta\left(\beta\vert\cdot\right)+\frac{1}{2}id_{\mathfrak{h}^{\ast}}\right)\\
 &  & -\left(-\beta\left(\beta\vert\cdot\right)+\frac{1}{2}id_{\mathfrak{h}^{\ast}}\right)\left(-\alpha\left(\alpha\vert\cdot\right)+\frac{1}{2}id_{\mathfrak{h}^{\ast}}\right)\\
 & = & \alpha\left(\alpha\vert\beta\right)\left(\beta\vert\cdot\right)-\frac{1}{2}\alpha\left(\alpha\vert\cdot\right)-\frac{1}{2}\beta\left(\beta\vert\cdot\right)+\frac{1}{4}id_{\mathfrak{h}^{\ast}}\\
 &  & -\beta\left(\beta\vert\alpha\right)\left(\alpha\vert\cdot\right)+\frac{1}{2}\beta\left(\beta\vert\cdot\right)+\frac{1}{2}\alpha\left(\alpha\vert\cdot\right)-\frac{1}{4}id_{\mathfrak{h}^{\ast}}\\
 & = & 0.
\end{eqnarray*}
Moreover, for $(\alpha|\beta)=\mp1$ one has the following: 
\begin{eqnarray*}
\left\{ X(\alpha),X(\beta)\right\}  & = & \left(-\alpha\left(\alpha\vert\cdot\right)+\frac{1}{2}id_{\mathfrak{h}^{\ast}}\right)\left(-\beta\left(\beta\vert\cdot\right)+\frac{1}{2}id_{\mathfrak{h}^{\ast}}\right)\\
 &  & +\left(-\beta\left(\beta\vert\cdot\right)+\frac{1}{2}id_{\mathfrak{h}^{\ast}}\right)\left(-\alpha\left(\alpha\vert\cdot\right)+\frac{1}{2}id_{\mathfrak{h}^{\ast}}\right)\\
 & = & \alpha\left(\alpha\vert\beta\right)\left(\beta\vert\cdot\right)-\frac{1}{2}\alpha\left(\alpha\vert\cdot\right)-\frac{1}{2}\beta\left(\beta\vert\cdot\right)+\frac{1}{4}id_{\mathfrak{h}^{\ast}}\\
 &  & +\beta\left(\beta\vert\alpha\right)\left(\alpha\vert\cdot\right)-\frac{1}{2}\beta\left(\beta\vert\cdot\right)-\frac{1}{2}\alpha\left(\alpha\vert\cdot\right)+\frac{1}{4}id_{\mathfrak{h}^{\ast}}\\
 & = & \mp\alpha\left(\beta\vert\cdot\right)\mp\beta\left(\alpha\vert\cdot\right)-\alpha\left(\alpha\vert\cdot\right)-\beta\left(\beta\vert\cdot\right)+\frac{1}{2}id_{\mathfrak{h}^{\ast}}\\
 & = & -\left(\pm\alpha\left(\beta\vert\cdot\right)\pm\beta\left(\alpha\vert\cdot\right)+\alpha\left(\alpha\vert\cdot\right)+\beta\left(\beta\vert\cdot\right)\right)+\frac{1}{2}id_{\mathfrak{h}^{\ast}}\\
 & = & -\left(\alpha\pm\beta\right)\left(\alpha\pm\beta\vert\cdot\right)+\frac{1}{2}id_{\mathfrak{h}^{\ast}}\\
 & = & X\left(\alpha\pm\beta\right).  \qedhere
\end{eqnarray*}
\end{proof}

\begin{rem}
Note that the canonical (non-reduced geometric) Weyl group representation $\rho : W \to \mathrm{GL}(\mathfrak{h}^*)$ acts via $\rho(s_\alpha)(x) = x - (\alpha|x) \alpha$ and so one has $X(\alpha) = \rho(s_\alpha) - \frac{1}{2} \mathrm{id}$. Therefore Remark~\ref{Paul} applies and the statement of Proposition~\ref{32coordinatefree} in fact follows from the observation that $\rho$ (restricted to any standard subgroup $\mathrm{Sym}_3)$ does not contain the sign representation as an irreducible component.  
\end{rem}

\section{Towards $\frac{5}{2}$-spin representations}

\begin{defn}
  Let $(T_{ab})_{a,b} \in \mathbb{R}^{k \times k}$ and $(U_{ab})_{a,b} \in \mathbb{R}^{k \times l}$.
  Then
  \[ \left(T_{(ab)}\right)_{a,b} \in \mathbb{R}^{k \times k}
  \]
  denotes the matrix with components
\[
T_{(ab)}:=\frac{1}{2}T_{ab}+\frac{1}{2}T_{ba}.
\]
Moreover, for $1 \leq a, b, c \leq k$ and $1 \leq d \leq l$ define
\[ T_{a(b}U_{c)d} := \frac{1}{2} T_{ab}U_{cd} + \frac{1}{2} T_{ac}U_{bd}. \]
This notation is called the {\em symmetrizer bracket}.
\end{defn}

\begin{lem}
As in Section~\ref{3/2} let $v^1$, ..., $v^k$ be a basis of $\mathfrak{h}^*$, let $(G^{ab})_{a,b}$ be the Gram matrix of the invariant form with respect to this basis, let $(G_{ab})_{a,b}$ be its inverse, let $\alpha=\sum_{i=1}^k\alpha_iv^i, \beta=\sum_{i=1}^k\beta_iv^i \in \mathfrak{h}^*$, and let $$\alpha^i := \sum_{j=1}^k G^{ij}\alpha_j, \quad\quad \beta^i := \sum_{j=1}^k G^{ij}\beta_j.$$ Then the following identities hold:
\begin{eqnarray}
\sum_{g,h=1}^k\alpha^{g}\alpha^{h}G_{g(c}G_{d)h} & = &\alpha_{c}\alpha_{d}=\alpha_{(c}\alpha_{d)} \label{I} \\
\sum_{g,h=1}^k\alpha^{g}\alpha^{h}\beta_{(g}G_{h)(c}\beta_{d)} & = &\left(\alpha\vert\beta\right)\alpha_{(c}\beta_{d)} \label{II}\\
\sum_{e,f,g,h=1}^k\alpha_{(a}G_{b)(e}\alpha_{f)}G^{eg}G^{fh}\beta_{(g}G_{h)(c}\beta_{d)} & = &\frac{1}{2}\alpha_{(a}\beta_{b)}\alpha_{(c}\beta_{d)}+\frac{1}{2}\left(\alpha\vert\beta\right)\alpha_{(a}G_{b)(c}\beta_{d)} \label{III}\\
\sum_{e,f,g,h=1}^kG_{a(e}G_{f)b}G^{eg}G^{fh}\beta_{(g}G_{h)(c}\beta_{d)} & = &\beta_{(a}G_{b)(c}\beta_{d)} \label{IV} \\
\sum_{e,f,g,h=1}^kG_{a(e}G_{f)b}G^{eg}G^{fh}G_{g(c}G_{d)h} & = & G_{a(c}G_{d)b}.  \label{V}
\end{eqnarray}
\end{lem}

\begin{proof}
Observe first that $$\sum_{g=1}^k \alpha^g G_{gc} = \sum_{g,i=1}^k G^{gi}\alpha_i G_{gc} = \sum_{g,i=1}^kG_{cg}G^{gi}\alpha_i = \alpha_c$$ and $$\sum_{g=1}^k \alpha^g \beta_g = \sum_{g,i=1}^k G^{gi}\alpha_i \beta_g = \sum_{g,i=1}^k \alpha_iG^{ig} \beta_g = (\alpha|\beta).$$  Equality \eqref{I} can then be established as follows: 
\begin{eqnarray*}
  \sum_{g,h=1}^k\alpha^{g}\alpha^{h}G_{g(c}G_{d)h} & = & \frac{1}{2}\sum_{g,h,i,j=1}^kG^{gi}\alpha_iG^{hj}\alpha_j\left(G_{gc}G_{dh}+G_{gd}G_{ch}\right)\\
  & = & \frac{1}{2}\sum_{g,h,i,j=1}^k G_{cg}G^{gi}\alpha_iG_{dh}G^{hj}\alpha_j + G_{dg}G^{gi}\alpha_i G_{ch}G^{hj}\alpha_j\\
 & = & \frac{1}{2}\left(\alpha_{c}\alpha_{d}+\alpha_{d}\alpha_{c}\right)=\alpha_{(c}\alpha_{d)}=\alpha_{c}\alpha_{d}.
\end{eqnarray*}
A similar computation yields equality \eqref{II}:
\begin{eqnarray*}
\sum_{g,h=1}^k\alpha^{g}\alpha^{h}\beta_{(g}G_{h)(c}\beta_{d)} & = & \frac{1}{4}\sum_{g,h=1}^k\alpha^g\alpha^h\left(\beta_{g}G_{hc}\beta_{d}+\beta_{h}G_{gc}\beta_{d}+\beta_{g}G_{hd}\beta_{c}+\beta_{h}G_{gd}\beta_{c}\right)\\
 & = & \frac{1}{4}\left((\alpha|\beta)\alpha_{c}\beta_{d}+(\alpha|\beta)\alpha_{c}\beta_{d}+\left(\alpha\vert\beta\right)\alpha_{d}\beta_{c}+\left(\alpha\vert\beta\right)\alpha_{d}\beta_{c}\right)\\
 & = & \frac{1}{2}\left(\alpha\vert\beta\right)\left(\alpha_{c}\beta_{d}+\alpha_{d}\beta_{c}\right)=\left(\alpha\vert\beta\right)\alpha_{(c}\beta_{d)}.
\end{eqnarray*}
For equality ~\eqref{III} one computes the following:
\begin{eqnarray*}
 & & 16\sum_{e,f,g,h=1}^k\alpha_{(a}G_{b)(e}\alpha_{f)}G^{eg}G^{fh}\beta_{(g}G_{h)(c}\beta_{d)}\\
&= & \sum_{e,f,g,h=1}^k\left(\alpha_{a}G_{be}\alpha_{f}+\alpha_{b}G_{ae}\alpha_{f}+\alpha_{a}G_{bf}\alpha_{e}+\alpha_{b}G_{af}\alpha_{e}\right)G^{eg}G^{fh}\\
 && \left(\beta_{g}G_{hc}\beta_{d}+\beta_{h}G_{gc}\beta_{d}+\beta_{g}G_{hd}\beta_{c}+\beta_{h}G_{gd}\beta_{c}\right)\\
&= &\sum_{g,h=1}^k\left(\alpha_{a}\delta_{bg}\alpha^{h}+\alpha_{b}\delta_{ah}\alpha^{g}+\alpha_{a}\delta_{bh}\alpha^{g}+\alpha_{b}\delta_{ag}\alpha^{h}\right)\\
 && \left(\beta_{g}G_{hc}\beta_{d}+\beta_{h}G_{gc}\beta_{d}+\beta_{g}G_{hd}\beta_{c}+\beta_{h}G_{gd}\beta_{c}\right)\\
&= &  \alpha_{a}\beta_{b}\alpha_{c}\beta_{d}+{\color{red}\left(\alpha\vert\beta\right)\alpha_{a}G_{bc}\beta_{d}}+\alpha_{a}\beta_{b}\beta_{c}\alpha_{d}+{\color{green}\left(\alpha\vert\beta\right)\alpha_{a}G_{bd}\beta_{c}}\\
 && +{\color{blue}\left(\alpha\vert\beta\right)\alpha_{b}G_{ac}\beta_{d}}+\beta_{a}\alpha_{b}\alpha_{c}\beta_{d}+{\color{cyan}\left(\alpha\vert\beta\right)\alpha_{b}G_{ad}\beta_{c}}+\beta_{a}\alpha_{b}\beta_{c}\alpha_{d}\\
 && +{\color{red}\left(\alpha\vert\beta\right)\alpha_{a}G_{bc}\beta_{d}}+\alpha_{a}\beta_{b}\alpha_{c}\beta_{d}+{\color{green}\left(\alpha\vert\beta\right)\alpha_{a}G_{bd}\beta_{c}}+\alpha_{a}\beta_{b}\beta_{c}\alpha_{d}\\
 && +\beta_{a}\alpha_{b}\alpha_{c}\beta_{d}+{\color{blue}\left(\alpha\vert\beta\right)\alpha_{b}G_{ac}\beta_{d}}+\beta_{a}\alpha_{b}\beta_{c}\alpha_{d}+{\color{cyan}\left(\alpha\vert\beta\right)\alpha_{b}G_{ad}\beta_{c}}\\
&= &  2\left(\alpha_{a}\beta_{b}\alpha_{c}\beta_{d}+\alpha_{a}\beta_{b}\beta_{c}\alpha_{d}+\beta_{a}\alpha_{b}\alpha_{c}\beta_{d}+\beta_{a}\alpha_{b}\beta_{c}\alpha_{d}\right)\\
 && +2\left(\alpha\vert\beta\right)\left(\alpha_{a}G_{bc}\beta_{d}+\alpha_{a}G_{bd}\beta_{c}+\alpha_{b}G_{ac}\beta_{d}+\alpha_{b}G_{ad}\beta_{c}\right)\\
&= &  8\alpha_{(a}\beta_{b)}\alpha_{(c}\beta_{d)}+8\left(\alpha\vert\beta\right)\alpha_{(a}G_{b)(c}\beta_{d)}
\end{eqnarray*}
and, hence, 
\[
\sum_{e,f,g,h=1}^k\alpha_{(a}G_{b)(e}\alpha_{f)}G^{eg}G^{fh}\beta_{(g}G_{h)(c}\beta_{d)}=\frac{1}{2}\alpha_{(a}\beta_{b)}\alpha_{(c}\beta_{d)}+\frac{1}{2}\left(\alpha\vert\beta\right)\alpha_{(a}G_{b)(c}\beta_{d)}.
\]
Equality \eqref{IV} can be established as follows:
\begin{eqnarray*}
 && \sum_{e,f,g,h=1}^kG_{a(e}G_{f)b}G^{eg}G^{fh}\beta_{(g}G_{h)(c}\beta_{d)}\\
&= & \frac{1}{8}\sum_{e,f,g,h=1}^k\left(G_{ae}G_{fb}+G_{af}G_{eb}\right)G^{eg}G^{fh}\left(\beta_{g}G_{hc}\beta_{d}+\beta_{h}G_{gc}\beta_{d}+\beta_{g}G_{hd}\beta_{c}+\beta_{h}G_{gd}\beta_{c}\right)\\
&= & \frac{1}{8}\sum_{g,h=1}^k\left(\delta_{ag}\delta_{bh}+\delta_{ah}\delta_{bg}\right)\left(\beta_{g}G_{hc}\beta_{d}+\beta_{h}G_{gc}\beta_{d}+\beta_{g}G_{hd}\beta_{c}+\beta_{h}G_{gd}\beta_{c}\right)\\
&= & \ \frac{1}{8}\left(\beta_{a}G_{bc}\beta_{d}+\beta_{b}G_{ac}\beta_{d}+\beta_{a}G_{bd}\beta_{c}+\beta_{b}G_{ad}\beta_{c}\right)\\
 && +\frac{1}{8}\left(\beta_{b}G_{ac}\beta_{d}+\beta_{a}G_{bc}\beta_{d}+\beta_{b}G_{ad}\beta_{c}+\beta_{a}G_{bd}\beta_{c}\right)\\
&= & \ \frac{1}{4}\left(\beta_{a}G_{bc}\beta_{d}+\beta_{b}G_{ac}\beta_{d}+\beta_{a}G_{bd}\beta_{c}+\beta_{b}G_{ad}\beta_{c}\right)\\
&= & \ \beta_{(a}G_{b)(c}\beta_{d)}.
\end{eqnarray*}
Finally, equality \eqref{V} can be shown as follows:
\begin{eqnarray*}
 && \sum_{e,f,g,h=1}^kG_{a(e}G_{f)b}G^{eg}G^{fh}G_{g(c}G_{d)h}\\
&= & \frac{1}{4}\sum_{g,h=1}^k\left(\delta_{ag}\delta_{bh}+\delta_{ah}\delta_{bg}\right)\left(G_{gc}G_{dh}+G_{gd}G_{ch}\right)\\
&= & \frac{1}{4}\left(G_{ac}G_{db}+G_{ad}G_{cb}+G_{bc}G_{da}+G_{bd}G_{ca}\right)\\
&= & \frac{1}{2}\left(G_{ac}G_{db}+G_{ad}G_{cb}\right)\\
&= & G_{a(c}G_{d)b}.   \qedhere
\end{eqnarray*}
\end{proof}

Throughout this section let $\mathfrak{g}$ be a simply
laced split real Kac--Moody algebra with maximal compact subalgebra $\mathfrak{k}$.
By Corollary~\ref{cor:Construction of a gen spin rep. for simply laced case}
there exists a generalized $\frac{1}{2}$-spin representation $\rho : \mathfrak{k}\rightarrow\mathrm{End}(\mathbb{C}^{l})$. In analogy to Sections~\ref{sub:Extending-a-generalized} and \ref{3/2} we make use of Clifford algebras in order to define higher spin representations.

Define $V := \mathrm{Sym}^2(\mathfrak{h}^*)$. Then, given a basis $v^1$, ..., $v^k$ of $\mathfrak{h}^*$, the vector space $V$ admits the natural basis $\{ v^{i_1} \otimes v^{i_2} \mid 1 \leq i_1 \leq i_2 \leq k \}$. Given an orthonormal basis $f_1$, ..., $f_s$ of $S$ as in Section~\ref{sub:Extending-a-generalized} one arrives at a basis $\{ v^{i_1} \otimes v^{i_2} \otimes f_j \mid 1 \leq i_1 \leq i_2 \leq k, 1 \leq j \leq l \}$ of $V \otimes S$.

In analogy to \eqref{phidef} define
\begin{eqnarray}
\phi^{ab}_\alpha := v^a \otimes v^b \otimes f_\alpha = v^b \otimes v^a \otimes f_\alpha =: \phi^{ba}_\alpha
\end{eqnarray}

The invariant symmetric bilinear form $(\cdot|\cdot)$ on $\mathfrak{h}^*$ induces a natural symmetric bilinear form on $\mathfrak{h}^* \otimes \mathfrak{h}^*$ which, by symmetry, factors through a symmetric bilinear form $q_1$ on $V = \mathrm{Sym}^2(\mathfrak{h}^*)$.
If $(G^{ab})_{1 \leq a, b \leq k}$ as in Section~\ref{3/2} denotes the Gram matrix of $(\cdot|\cdot)$ with respect to the basis $v^1$, ..., $v^k$, the computation
\begin{eqnarray*}
q_1\left(v^a \otimes v^b,v^c \otimes v^d\right)  & = & \frac{1}{4}q_1\left(v^{a}\otimes v^{b}+v^{b}\otimes v^{a},v^{c}\otimes v^{d}+v^{d}\otimes v^{c}\right)\\
 & = & \frac{1}{4}\left[q_1\left(v^{a}\otimes v^{b},v^{c}\otimes v^{d}\right)+q_1\left(v^{a}\otimes v^{b},v^{d}\otimes v^{c}\right)\right]\\
 &  & +\frac{1}{4}\left[q_1\left(v^{b}\otimes v^{a},v^{c}\otimes v^{d}\right)+q_1\left(v^{b}\otimes v^{a},v^{d}\otimes v^{c}\right)\right]\\
 & = & \frac{1}{4}\left(G^{ac}G^{bd}+G^{ad}G^{bc}+G^{bc}G^{ad}+G^{bd}G^{ac}\right)\\
& = & \frac{1}{2}\left(G^{ac}G^{db}+G^{ad}G^{cb}\right) = \frac{1}{2}\left(G^{bd}G^{ca}+G^{bc}G^{da}\right) = \frac{1}{2}\left(G^{ca}G^{bd}+G^{cb}G^{ad}\right) \\
& = & G^{a(c}G^{d)b} = G^{b(c}G^{d)a} = G^{b(d}G^{c)a} = G^{c(a}G^{b)d}
\end{eqnarray*}
shows that the various symmetrizer brackets $$G^{a(c}G^{d)b} = G^{b(c}G^{d)a} = G^{b(d}G^{c)a} = G^{c(a}G^{b)d} = q_1\left(v^a \otimes v^b,v^c \otimes v^d\right) = G^{ac}G^{bd} \quad\quad \text{for all $a,b,c,d \in \{ 1, ..., k \}$}$$
all describe the Gram matrix of $q_1$ with respect to the basis $\{ v^{i_1} \otimes v^{i_2} \mid 1 \leq i_1 \leq i_2 \leq k \}$.
In analogy to Section~\ref{sub:Extending-a-generalized} define a symmetric bilinear form on the tensor product $V \otimes S$ via
$$q(\phi^{ab}_\alpha,\phi^{cd}_\beta) := G^{a(c}G^{d)b}\delta_{\alpha\beta} = G^{ac}G^{bd}\delta_{\alpha\beta}.$$
The above equality between various symmetrizer brackets makes it meaningful to define $$\phi^{\mathcal{A}}_\alpha := \frac{1}{2} \phi^{ab}_\alpha + \frac{1}{2} \phi^{ba}_\alpha \quad\quad \phi^{\mathcal{B}}_\beta := \frac{1}{2} \phi^{cd}_\beta + \frac{1}{2} \phi^{dc}_\beta$$
and $$G^{\mathcal{A}\mathcal{B}} := G^{a(c}G^{d)b} = G^{ac}G^{bd} = G^{ca}G^{db} = G^{c(a}G^{b)d} =: G^{\mathcal{B}\mathcal{A}}$$
in order to make the formalism of Lemma~\ref{lem:commutator of bilinear ansatz} and Proposition~\ref{prop:criterion for root operator-1} applicable also in the situation of $V = \mathrm{Sym}^2(\mathfrak{h}^*)$ by interpreting $\mathcal{A}$ and $\mathcal{B}$ as multi-indices whose constituents vary independently between $1$ and $k$.
For instance, expanding the commutator
\[
\left[X,Y\right]_{\mathcal{AB}} = \sum_{\mathcal{C},\mathcal{D}=1}^k\left(X_{\mathcal{AC}}G^{\mathcal{CD}}Y_{\mathcal{DB}}-Y_{\mathcal{AC}}G^{\mathcal{CD}}X_{\mathcal{DB}}\right)
\]
from
\eqref{so91commutator}
into the current setting with
\begin{align*}
  \mathcal{A} & \text{ corresponding to $a$, $b$}, \\
  \mathcal{B} & \text{ corresponding to $c$, $d$}, \\
  \mathcal{C} & \text{ corresponding to $e$, $f$}, \\
  \mathcal{D} & \text{ corresponding to $g$, $h$}
\end{align*}
  then yields
\[
\left[X,Y\right]_{ab\, cd}=\sum_{e,f,g,h=1}^k\left(X_{ab\, ef}G^{eg}G^{fh}Y_{gh\, cd}-Y_{ab\, ef}G^{eg}G^{fh}X_{gh\, cd}\right).
\]

\begin{prop} \label{5halfspin}
Let $\alpha=\sum_{i=1}^k\alpha_iv^i \in \mathfrak{h}^*$ be a real root. Then the matrices given by 
\begin{eqnarray}
X(\alpha)_{ab\, cd}=\frac{1}{2}\alpha_{a}\alpha_{b}\alpha_{c}\alpha_{d}-\alpha_{(a}G_{b)(c}\alpha_{d)}+\frac{1}{4}G_{a(c}G_{d)b} \label{52def}
\end{eqnarray}
satisfy for all $\alpha,\beta\in\Delta^{\mathrm{re}}$ such that $\left(\alpha\vert\beta\right)=0$
\[
\left[X(\alpha),X(\beta)\right]_{ab\, cd}=\sum_{e,f,g,h=1}^k\left(X(\alpha)_{ab\, ef}G^{eg}G^{fh}X(\beta)_{gh\, cd}-X(\beta)_{ab\, ef}G^{eg}G^{fh}X(\alpha)_{gh\, cd}\right)=0
\]
and for all $\alpha,\beta\in\Delta^{\mathrm{re}}$ such that $\left(\alpha\vert\beta\right)=\mp1$
\begin{eqnarray*}
\left\{ X(\alpha),X(\beta)\right\}_{ab\, cd}  & = & \sum_{e,f,g,h=1}^k\left(X(\alpha)_{ab\, ef}G^{eg}G^{fh}X(\beta)_{gh\, cd}+X(\beta)_{ab\, ef}G^{eg}G^{fh}X(\alpha)_{gh\, cd}\right)\\
 & = & \frac{1}{2}X(\alpha\pm\beta).
\end{eqnarray*}
In particular, the assigment $X_i \mapsto X(\alpha_i)$ defines a finite-dimensional representation of $\mathfrak{k}$.
\end{prop}

\begin{rem}
Formula \eqref{52def} is \cite[(5.4), p.~18]{AN1}.
\end{rem}

\begin{proof}[Proof of Proposition~\ref{5halfspin}.]
It suffices to establish the hypotheses of Proposition~\ref{prop:criterion for root operator-1}. By definition, $X(\alpha)$ is symmetric.

Define $$S_{ab\, cd}:=\sum_{e,f,g,h=1}^kX(\alpha)_{ab\, ef}G^{eg}G^{fh}X(\beta)_{gh\, cd}$$
and calculate the following; for the sake of the exposition in the next calculation we use Einstein's summation convention, i.e., equal indices are summed over if one is upper and one is lower.
\begin{eqnarray*}
S_{ab\, cd} & = & \left(\frac{1}{2}\alpha_{a}\alpha_{b}\alpha_{e}\alpha_{f}-\alpha_{(a}G_{b)(e}\alpha_{f)}+\frac{1}{4}G_{a(e}G_{f)b}\right)G^{eg}G^{fh}\\
 &  & \left(\frac{1}{2}\beta_{g}\beta_{h}\beta_{c}\beta_{d}-\beta_{(g}G_{h)(c}\beta_{d)}+\frac{1}{4}G_{g(c}G_{d)h}\right)\\
 & = & \frac{1}{4}\alpha_{a}\alpha_{b}\alpha^{g}\alpha^{h}\beta_{g}\beta_{h}\beta_{c}\beta_{d}-\frac{1}{2}\alpha_{a}\alpha_{b}\alpha^{g}\alpha^{h}\beta_{(g}G_{h)(c}\beta_{d)}+\frac{1}{8}\alpha_{a}\alpha_{b}\alpha^{g}\alpha^{h}G_{g(c}G_{d)h}\\
 &  & -\frac{1}{2}\alpha_{(a}G_{b)(e}\alpha_{f)}\beta^{e}\beta^{f}\beta_{c}\beta_{d}+\alpha_{(a}G_{b)(e}\alpha_{f)}G^{eg}G^{fh}\beta_{(g}G_{h)(c}\beta_{d)}\\
 &  & -\frac{1}{4}\alpha_{(a}G_{b)(e}\alpha_{f)}G^{eg}G^{fh}G_{g(c}G_{d)h}+\frac{1}{8}G_{a(e}G_{f)b}\beta^{e}\beta^{f}\beta_{c}\beta_{d}\\
 &  & -\frac{1}{4}G_{a(e}G_{f)b}G^{eg}G^{fh}\beta_{(g}G_{h)(c}\beta_{d)}+\frac{1}{16}G_{a(e}G_{f)b}G^{eg}G^{fh}G_{g(c}G_{d)h}\\
 & = & \frac{1}{4}\left(\alpha\vert\beta\right)^{2}\alpha_{a}\alpha_{b}\beta_{c}\beta_{d}-\frac{1}{2}\left(\alpha\vert\beta\right)\alpha_{a}\alpha_{b}\alpha_{(c}\beta_{d)}+\frac{1}{8}\alpha_{a}\alpha_{b}\alpha_{c}\alpha_{d}\\
 &  & -\frac{1}{2}\left(\alpha\vert\beta\right)\alpha_{(a}\beta_{b)}\beta_{c}\beta_{d}+\frac{1}{2}\alpha_{(a}\beta_{b)}\alpha_{(c}\beta_{d)}+\frac{1}{2}\left(\alpha\vert\beta\right)\alpha_{(a}G_{b)(c}\beta_{d)}\\
 &  & -\frac{1}{4}\alpha_{(a}G_{b)(c}\alpha_{d)}+\frac{1}{8}\beta_{a}\beta_{b}\beta_{c}\beta_{d}-\frac{1}{4}\beta_{(a}G_{b)(c}\beta_{d)}+\frac{1}{16}G_{a(c}G_{d)b}.
\end{eqnarray*}
Next one computes the commutator $C_{ab\, cd}:=\left[X(\alpha),X(\beta)\right]_{ab\, cd}$
in the case of $\left(\alpha\vert\beta\right)=0$ to be equal to 
\begin{eqnarray*}
C_{ab\, cd} & = & \sum_{e,f,g,h=1}^k \left(X(\alpha)_{ab\, ef}G^{eg}G^{fh}X(\beta)_{gh\, cd}-X(\beta)_{ab\, ef}G^{eg}G^{fh}X(\alpha)_{gh\, cd}\right)\\
 & = & \frac{1}{8}\alpha_{a}\alpha_{b}\alpha_{c}\alpha_{d}+\frac{1}{2}\alpha_{(a}\beta_{b)}\alpha_{(c}\beta_{d)}-\frac{1}{4}\alpha_{(a}G_{b)(c}\alpha_{d)}\\
 &  & +\frac{1}{8}\beta_{a}\beta_{b}\beta_{c}\beta_{d}-\frac{1}{4}\beta_{(a}G_{b)(c}\beta_{d)}+\frac{1}{16}G_{a(c}G_{d)b}-(\alpha\leftrightarrow\beta)\\
 & = & \frac{1}{8}{\color{red}\alpha_{a}\alpha_{b}\alpha_{c}\alpha_{d}}+{\color{green}\frac{1}{2}\alpha_{(a}\beta_{b)}\alpha_{(c}\beta_{d)}}-{\color{blue}\frac{1}{4}\alpha_{(a}G_{b)(c}\alpha_{d)}}\\
 &  & +\frac{1}{8}{\color{cyan}\beta_{a}\beta_{b}\beta_{c}\beta_{d}}-{\color{magenta}\frac{1}{4}\beta_{(a}G_{b)(c}\beta_{d)}}+\frac{1}{16}G_{a(c}G_{d)b}\\
 &  & -\frac{1}{8}{\color{cyan}\beta_{a}\beta_{b}\beta_{c}\beta_{d}}-{\color{green}\frac{1}{2}\beta_{(a}\alpha_{b)}\beta_{(c}\alpha_{d)}}+{\color{magenta}\frac{1}{4}\beta_{(a}G_{b)(c}\beta_{d)}}\\
 &  & -\frac{1}{8}{\color{red}\alpha_{a}\alpha_{b}\alpha_{c}\alpha_{d}}+{\color{blue}\frac{1}{4}\alpha_{(a}G_{b)(c}\alpha_{d)}}-\frac{1}{16}G_{a(c}G_{d)b}\\
 & = & 0,
\end{eqnarray*}
since $\alpha_{(a}\beta_{b)}\alpha_{(c}\beta_{d)}=\beta_{(a}\alpha_{b)}\beta_{(c}\alpha_{d)}$
and the terms with matching colours cancel. (Here the symbol $(\alpha\leftrightarrow\beta)$ denotes a repetition of all previous terms with the roles of $\alpha$ and $\beta$ interchanged.) In a similar fashion one
calculates the anti-commutator $$A_{ab\, cd}:=\sum_{e,f,g,h=1}^k\left(X(\alpha)_{ab\, ef}G^{eg}G^{fh}X(\beta)_{gh\, cd}+X(\beta)_{ab\, ef}G^{eg}G^{fh}X(\alpha)_{gh\, cd}\right)$$
for $\left(\alpha\vert\beta\right)=\pm1$ to be 
\begin{eqnarray*}
A_{ab\, cd} & = & \frac{1}{4}\alpha_{a}\alpha_{b}\beta_{c}\beta_{d}\mp\frac{1}{2}\alpha_{a}\alpha_{b}\alpha_{(c}\beta_{d)}+\frac{1}{8}\alpha_{a}\alpha_{b}\alpha_{c}\alpha_{d}\\
 &  & \mp\frac{1}{2}\alpha_{(a}\beta_{b)}\beta_{c}\beta_{d}+\frac{1}{2}\alpha_{(a}\beta_{b)}\alpha_{(c}\beta_{d)}\pm\frac{1}{2}\alpha_{(a}G_{b)(c}\beta_{d)}\\
 &  & -\frac{1}{4}\alpha_{(a}G_{b)(c}\alpha_{d)}+\frac{1}{8}\beta_{a}\beta_{b}\beta_{c}\beta_{d}-\frac{1}{4}\beta_{(a}G_{b)(c}\beta_{d)}+\frac{1}{16}G_{a(c}G_{d)b}\\
 &  & \frac{1}{4}\beta_{a}\beta_{b}\alpha_{c}\alpha_{d}\mp\frac{1}{2}\beta_{a}\beta_{b}\beta_{(c}\alpha_{d)}+\frac{1}{8}\beta_{a}\beta_{b}\beta_{c}\beta_{d}\\
 &  & \mp\frac{1}{2}\beta_{(a}\alpha_{b)}\alpha_{c}\alpha_{d}+\frac{1}{2}\beta_{(a}\alpha_{b)}\beta_{(c}\alpha_{d)}\pm\frac{1}{2}\beta_{(a}G_{b)(c}\alpha_{d)}\\
 &  & -\frac{1}{4}\beta_{(a}G_{b)(c}\beta_{d)}+\frac{1}{8}\alpha_{a}\alpha_{b}\alpha_{c}\alpha_{d}-\frac{1}{4}\alpha_{(a}G_{b)(c}\alpha_{d)}+\frac{1}{16}G_{a(c}G_{d)b} \\
 & = & \frac{1}{4}\alpha_{a}\alpha_{b}\alpha_{c}\alpha_{d}\mp\frac{1}{2}\alpha_{a}\alpha_{b}\alpha_{(c}\beta_{d)}\mp\frac{1}{2}\beta_{(a}\alpha_{b)}\alpha_{c}\alpha_{d}+\frac{1}{4}\alpha_{a}\alpha_{b}\beta_{c}\beta_{d}\\
 &  & +\frac{1}{4}\beta_{a}\beta_{b}\alpha_{c}\alpha_{d}+\alpha_{(a}\beta_{b)}\alpha_{(c}\beta_{d)}\mp\frac{1}{2}\beta_{a}\beta_{b}\beta_{(c}\alpha_{d)}\mp\frac{1}{2}\alpha_{(a}\beta_{b)}\beta_{c}\beta_{d}\\
 &  & +\frac{1}{4}\beta_{a}\beta_{b}\beta_{c}\beta_{d}\\
 &  & -\frac{1}{2}\alpha_{(a}G_{b)(c}\alpha_{d)}\pm\frac{1}{2}\alpha_{(a}G_{b)(c}\beta_{d)}\pm\frac{1}{2}\beta_{(a}G_{b)(c}\alpha_{d)}-\frac{1}{2}\beta_{(a}G_{b)(c}\beta_{d)}\\
 &  & +\frac{1}{8}G_{a(c}G_{d)b} \\
 & = & \frac{1}{4}\alpha_{a}\alpha_{b}\alpha_{c}\alpha_{d}\mp\frac{1}{4}\left(\alpha_{a}\alpha_{b}\alpha_{c}\beta_{d}+\alpha_{a}\alpha_{b}\beta_{c}\alpha_{d}+\alpha_{a}\beta_{b}\alpha_{c}\alpha_{d}+\beta_{a}\alpha_{b}\alpha_{c}\alpha_{d}\right)\\
 &  & +\frac{1}{4}\left(\alpha_{a}\alpha_{b}\beta_{c}\beta_{d}+\beta_{a}\beta_{b}\alpha_{c}\alpha_{d}+\alpha_{a}\beta_{b}\alpha_{c}\beta_{d}+\beta_{a}\alpha_{b}\alpha_{c}\beta_{d}+\alpha_{a}\beta_{b}\beta_{c}\alpha_{d}+\beta_{a}\alpha_{b}\beta_{c}\alpha_{d}\right)\\
 &  & \mp\frac{1}{4}\left(\beta_{a}\beta_{b}\beta_{c}\alpha_{d}+\beta_{a}\beta_{b}\alpha_{c}\beta_{d}+\beta_{a}\alpha_{b}\beta_{c}\beta_{d}+\alpha_{a}\beta_{b}\beta_{c}\beta_{d}\right)+\frac{1}{4}\beta_{a}\beta_{b}\beta_{c}\beta_{d}\\
 &  & -\frac{1}{2}\left(\alpha_{(a}G_{b)(c}\alpha_{d)}\mp\alpha_{(a}G_{b)(c}\beta_{d)}\mp\beta_{(a}G_{b)(c}\alpha_{d)}+\beta_{(a}G_{b)(c}\beta_{d)}\right)+\frac{1}{8}G_{a(c}G_{d)b}\\
 & = & \frac{1}{4}(\alpha\mp\beta)_{a}(\alpha\mp\beta)_{b}(\alpha\mp\beta)_{c}(\alpha\mp\beta)_{d}-\frac{1}{2}(\alpha\mp\beta)_{(a}G_{b)(c}(\alpha\mp\beta)_{d)}+\frac{1}{8}G_{a(c}G_{d)b}\\
 & = & \frac{1}{2}X(\alpha\mp\beta)_{ab\, cd}.
\end{eqnarray*}
This proves the claim.
\end{proof}

Again, we conclude this section with a coordinate-free version of Proposition~\ref{5halfspin}. 

\begin{prop} \label{free52}
For $V=\mathfrak{h}^{\ast}$ let $\left(\cdot\vert\cdot\right)$ denote
the induced invariant bilinear form on $\mathfrak{h}^{\ast}$. Moreover, for $\alpha \in \Delta^{\mathrm{re}}$ let $\pi_{\alpha}:=\alpha\left(\alpha\vert\cdot\right)\in\mathrm{End}\left(\mathfrak{h}^{\ast}\right)$.
Define
$X : \Delta^{\mathrm{re}}\rightarrow\mathrm{End}\left(\mathrm{Sym}^2(\mathfrak{h}^{\ast})\right)$
via 
\begin{eqnarray}
\alpha\mapsto X(\alpha) & := & \pi_{\alpha}\otimes\pi_{\alpha}-\left(\pi_{\alpha}\otimes id_{\mathfrak{h}^{\ast}}+id_{\mathfrak{h}^{\ast}}\otimes\pi_{\alpha}\right)+\frac{1}{2}id_{\mathfrak{h}^{\ast}}\otimes id_{\mathfrak{h}^{\ast}}.    \label{formula52}
\end{eqnarray}
Then $X$ satisfies \eqref{eq:commutator for X as endo}  and
\eqref{eq:anticommutator for X as endo} for all real roots $\alpha,\beta$
with $\left(\alpha\vert\beta\right)\in\left\{ 0,\pm1\right\} $ and
thus provides a representation $\sigma$ of $\mathfrak{k}$
by
sending 
\[
X_{i}\mapsto\sigma\left(X_{i}\right):=X\left(\alpha_{i}\right)\otimes\Gamma\left(\alpha_{i}\right).
\]
\end{prop}

\begin{proof}
Observe 
\[
\pi_{\alpha}\pi_{\beta}=\begin{cases}
0, & \text{if }\left(\alpha\vert\beta\right)=0,\\
\pm\alpha\left(\beta\vert\cdot\right), & \text{if }\left(\alpha\vert\beta\right)=\pm1,
\end{cases}
\]
and abbreviate $_{\alpha}\pi_{\beta}:=\alpha\left(\beta\vert\cdot\right)$
and $1\equiv id_{\mathfrak{h}^{\ast}}$. One computes for $\left(\alpha\vert\beta\right)=0$
that 
\begin{eqnarray*}
\left[X(\alpha),X(\beta)\right] & = & \pi_{\alpha}\pi_{\beta}\otimes\pi_{\alpha}\pi_{\beta}-\pi_{\alpha}\pi_{\beta}\otimes\pi_{\alpha}-\pi_{\alpha}\otimes\pi_{\alpha}\pi_{\beta}+\frac{1}{2}\pi_{\alpha}\otimes\pi_{\alpha}\\
 &  & -\pi_{\alpha}\pi_{\beta}\otimes\pi_{\beta}+\pi_{\alpha}\pi_{\beta}\otimes1+\pi_{\alpha}\otimes\pi_{\beta}-\frac{1}{2}\pi_{\alpha}\otimes1\\
 &  & -\pi_{\beta}\otimes\pi_{\alpha}\pi_{\beta}+\pi_{\beta}\otimes\pi_{\alpha}+1\otimes\pi_{\alpha}\pi_{\beta}-\frac{1}{2}\cdot1\otimes\pi_{\alpha}\\
 &  & +\frac{1}{2}\pi_{\beta}\otimes\pi_{\beta}-\frac{1}{2}\left(\pi_{\beta}\otimes1+1\otimes\pi_{\beta}\right)+\frac{1}{4}\cdot1\otimes1\\
 &  & -\left(\alpha\leftrightarrow\beta\right)\\
 & = & \frac{1}{2}\pi_{\alpha}\otimes\pi_{\alpha}+\pi_{\alpha}\otimes\pi_{\beta}-\frac{1}{2}\pi_{\alpha}\otimes1+\pi_{\beta}\otimes\pi_{\alpha}\\
 &  & -\frac{1}{2}\cdot1\otimes\pi_{\alpha}+\frac{1}{2}\pi_{\beta}\otimes\pi_{\beta}-\frac{1}{2}\left(\pi_{\beta}\otimes1+1\otimes\pi_{\beta}\right)+\frac{1}{4}\cdot1\otimes1\\
 &  & -\left(\alpha\leftrightarrow\beta\right)\\
 & = & 0
\end{eqnarray*}
because the first part is symmetric in $\alpha$ and $\beta$. (Here the symbol $(\alpha\leftrightarrow\beta)$ again denotes a repetition of all previous terms with the roles of $\alpha$ and $\beta$ interchanged.)

Before
evaluating the anti-commutator consider for $(\alpha\vert\beta)=\mp1$
\begin{eqnarray*}
\pi_{\alpha\pm\beta} & = & \left(\alpha\pm\beta\right)\left(\alpha\pm\beta\vert\cdot\right)\\
 & = & \pi_{\alpha}\pm{}_{\alpha}\pi_{\beta}\pm{}_{\beta}\pi_{\alpha}+\pi_{\beta}\\
 & = & \pi_{\alpha}-\pi_{\alpha}\pi_{\beta}-\pi_{\beta}\pi_{\alpha}+\pi_{\beta}
\end{eqnarray*}
 and, thus, 
\begin{eqnarray}
\pi_{\alpha\pm\beta}\otimes\pi_{\alpha\pm\beta} & = & \left(\pi_{\alpha}\pm{}_{\alpha}\pi_{\beta}\pm{}_{\beta}\pi_{\alpha}+\pi_{\beta}\right)\otimes\left(\pi_{\alpha}\pm{}_{\alpha}\pi_{\beta}\pm{}_{\beta}\pi_{\alpha}+\pi_{\beta}\right) \notag \\
 & = & \left(\pi_{\alpha}-\pi_{\alpha}\pi_{\beta}-\pi_{\beta}\pi_{\alpha}+\pi_{\beta}\right)\otimes\left(\pi_{\alpha}-\pi_{\alpha}\pi_{\beta}-\pi_{\beta}\pi_{\alpha}+\pi_{\beta}\right) \notag \\
 & = & \pi_{\alpha}\otimes\pi_{\alpha}-\pi_{\alpha}\otimes\pi_{\alpha}\pi_{\beta}-\pi_{\alpha}\otimes\pi_{\beta}\pi_{\alpha}+\pi_{\alpha}\otimes\pi_{\beta} \notag \\
 &  & -\pi_{\alpha}\pi_{\beta}\otimes\pi_{\alpha}+\pi_{\alpha}\pi_{\beta}\otimes\pi_{\alpha}\pi_{\beta}+{\color{magenta}\pi_{\alpha}\pi_{\beta}\otimes\pi_{\beta}\pi_{\alpha}}-\pi_{\alpha}\pi_{\beta}\otimes\pi_{\beta} \notag \\
 &  & -\pi_{\beta}\pi_{\alpha}\otimes\pi_{\alpha}+{\color{magenta}\pi_{\beta}\pi_{\alpha}\otimes\pi_{\alpha}\pi_{\beta}}+\pi_{\beta}\pi_{\alpha}\otimes\pi_{\beta}\pi_{\alpha}-\pi_{\beta}\pi_{\alpha}\otimes\pi_{\beta} \notag \\
 &  & +\pi_{\beta}\otimes\pi_{\alpha}-\pi_{\beta}\otimes\pi_{\alpha}\pi_{\beta}-\pi_{\beta}\otimes\pi_{\beta}\pi_{\alpha}+\pi_{\beta}\otimes\pi_{\beta}.  \label{longformula}
\end{eqnarray}
For the anti-commutator one computes
\begin{eqnarray*}
\left\{ X(\alpha),X(\beta)\right\}  & = & {\color{red}\pi_{\alpha}\pi_{\beta}\otimes\pi_{\alpha}\pi_{\beta}}-{\color{red}\pi_{\alpha}\pi_{\beta}\otimes\pi_{\alpha}}-{\color{red}\pi_{\alpha}\otimes\pi_{\alpha}\pi_{\beta}}+{\color{red}\frac{1}{2}\pi_{\alpha}\otimes\pi_{\alpha}}\\
 &  & -{\color{red}\pi_{\alpha}\pi_{\beta}\otimes\pi_{\beta}}+{\color{green}\pi_{\alpha}\pi_{\beta}\otimes1}+{\color{red}\pi_{\alpha}\otimes\pi_{\beta}}{\color{green}-\frac{1}{2}\pi_{\alpha}\otimes1}\\
 &  & -{\color{red}\pi_{\beta}\otimes\pi_{\alpha}\pi_{\beta}}+{\color{red}\pi_{\beta}\otimes\pi_{\alpha}}+{\color{green}1\otimes\pi_{\alpha}\pi_{\beta}}-{\color{green}\frac{1}{2}\cdot1\otimes\pi_{\alpha}}\\
 &  & +{\color{red}\frac{1}{2}\pi_{\beta}\otimes\pi_{\beta}}-{\color{green}\frac{1}{2}\left(\pi_{\beta}\otimes1+1\otimes\pi_{\beta}\right)}+\frac{1}{4}\cdot1\otimes1\\
 &  & +\left(\alpha\leftrightarrow\beta\right)\\
 & = & {\color{red}\pi_{\alpha}\pi_{\beta}\otimes\pi_{\alpha}\pi_{\beta}-\pi_{\alpha}\pi_{\beta}\otimes\pi_{\alpha}-\pi_{\alpha}\otimes\pi_{\alpha}\pi_{\beta}-\pi_{\alpha}\pi_{\beta}\otimes\pi_{\beta}}\\
 &  & {\color{red}+\pi_{\alpha}\otimes\pi_{\beta}-\pi_{\beta}\otimes\pi_{\alpha}\pi_{\beta}+\pi_{\beta}\otimes\pi_{\alpha}+\frac{1}{2}\pi_{\beta}\otimes\pi_{\beta}+\frac{1}{2}\pi_{\alpha}\otimes\pi_{\alpha}}\\
 &  & +\pi_{\beta}\pi_{\alpha}\otimes\pi_{\beta}\pi_{\alpha}-\pi_{\beta}\pi_{\alpha}\otimes\pi_{\beta}-\pi_{\beta}\otimes\pi_{\beta}\pi_{\alpha}-\pi_{\beta}\pi_{\alpha}\otimes\pi_{\alpha}\\
 &  & +\pi_{\beta}\otimes\pi_{\alpha}-\pi_{\alpha}\otimes\pi_{\beta}\pi_{\alpha}+\pi_{\alpha}\otimes\pi_{\beta}+\frac{1}{2}\pi_{\alpha}\otimes\pi_{\alpha}+\frac{1}{2}\pi_{\beta}\otimes\pi_{\beta}\\
 &  & {\color{green}+\pi_{\alpha}\pi_{\beta}\otimes1-\frac{1}{2}\pi_{\alpha}\otimes1+1\otimes\pi_{\alpha}\pi_{\beta}-\frac{1}{2}\cdot1\otimes\pi_{\alpha}}\\
 &  & {\color{green}-\frac{1}{2}\pi_{\beta}\otimes1-\frac{1}{2}1\otimes\pi_{\beta}}\\
 &  & +\pi_{\beta}\pi_{\alpha}\otimes1-\frac{1}{2}\pi_{\beta}\otimes1+1\otimes\pi_{\beta}\pi_{\alpha}-\frac{1}{2}1\otimes\pi_{\beta}\\
 &  & -\frac{1}{2}\pi_{\alpha}\otimes1-\frac{1}{2}\cdot1\otimes\pi_{\alpha}\\
 &  & +\frac{2}{4}\cdot1\otimes1.
\end{eqnarray*}
The two red lines and the consecutive two lines equal the term
\begin{eqnarray*}
 & & \pi_{\alpha}\otimes\pi_{\alpha}-\pi_{\alpha}\otimes\pi_{\alpha}\pi_{\beta}-\pi_{\alpha}\otimes\pi_{\beta}\pi_{\alpha}+\pi_{\alpha}\otimes\pi_{\beta}\\
 &  & -\pi_{\beta}\pi_{\alpha}\otimes\pi_{\alpha}-\pi_{\beta}\pi_{\alpha}\otimes\pi_{\beta}+\pi_{\beta}\pi_{\alpha}\otimes\pi_{\beta}\pi_{\alpha}+{\color{magenta}\pi_{\beta}\otimes\pi_{\alpha}}\\
 &  & -\pi_{\alpha}\pi_{\beta}\otimes\pi_{\alpha}+\pi_{\alpha}\pi_{\beta}\otimes\pi_{\alpha}\pi_{\beta}-\pi_{\alpha}\pi_{\beta}\otimes\pi_{\beta}+{\color{magenta}\pi_{\alpha}\otimes\pi_{\beta}}\\
 &  & \pi_{\beta}\otimes\pi_{\beta}-\pi_{\beta}\otimes\pi_{\beta}\pi_{\alpha}-\pi_{\beta}\otimes\pi_{\alpha}\pi_{\beta}+\pi_{\beta}\otimes\pi_{\alpha}
\end{eqnarray*}
which is almost identical to the expression for $\pi_{\alpha\pm\beta}\otimes\pi_{\alpha\pm\beta}$ derived in formula \eqref{longformula}
if it were not for the pink terms. Nevertheless, for $h\in\mathfrak{h}^{\ast}$
one evaluates
\begin{eqnarray*}
\left(\pi_{\alpha}\pi_{\beta}\otimes\pi_{\beta}\pi_{\alpha}+\pi_{\beta}\pi_{\alpha}\otimes\pi_{\alpha}\pi_{\beta}\right)\left(h,h\right) & = & \left(\mp\alpha\left(\beta\vert h\right)\right)\otimes\left(\mp\beta\left(\alpha\vert h\right)\right)\\
 &  & +\left(\mp\beta\left(\alpha\vert h\right)\right)\otimes\left(\mp\alpha\left(\beta\vert h\right)\right)\\
 & = & \left(\alpha\vert h\right)\left(\beta\vert h\right)\cdot\left(\alpha\otimes\beta+\beta\otimes\alpha\right)
\end{eqnarray*}
whereas 
\begin{eqnarray*}
\left(\pi_{\alpha}\otimes\pi_{\beta}+\pi_{\beta}\otimes\pi_{\alpha}\right)(h,h) & = & \alpha\left(\alpha\vert h\right)\otimes\left(\beta\vert h\right)\beta\\
 &  & +\left(\beta\vert h\right)\beta\otimes\alpha\left(\alpha\vert h\right)\\
 & = & \left(\alpha\vert h\right)\left(\beta\vert h\right)\cdot\left(\alpha\otimes\beta+\beta\otimes\alpha\right)\\
 & = & \left(\pi_{\alpha}\pi_{\beta}\otimes\pi_{\beta}\pi_{\alpha}+\pi_{\beta}\pi_{\alpha}\otimes\pi_{\alpha}\pi_{\beta}\right)\left(h,h\right)
\end{eqnarray*}
for arbitrary $h\in\mathfrak{h}^{\ast}$. Thus, using the diagonalizability of real symmetric tensors of degree two, the first four lines
of the anti-commutator are equal to $\pi_{\alpha\pm\beta}\otimes\pi_{\alpha\pm\beta}$.
The green lines and the two consecutive lines are evaluated to be 
\begin{eqnarray*}
 &  & {\color{green}+\pi_{\alpha}\pi_{\beta}\otimes1-\frac{1}{2}\pi_{\alpha}\otimes1+1\otimes\pi_{\alpha}\pi_{\beta}-\frac{1}{2}\cdot1\otimes\pi_{\alpha}}\\
 &  & {\color{green}-\frac{1}{2}\pi_{\beta}\otimes1-\frac{1}{2}1\otimes\pi_{\beta}}\\
 &  & +\pi_{\beta}\pi_{\alpha}\otimes1-\frac{1}{2}\pi_{\beta}\otimes1+1\otimes\pi_{\beta}\pi_{\alpha}-\frac{1}{2}1\otimes\pi_{\beta}\\
 &  & -\frac{1}{2}\pi_{\alpha}\otimes1-\frac{1}{2}\cdot1\otimes\pi_{\alpha}\\
 & = & \left(\pi_{\alpha}\pi_{\beta}-\frac{1}{2}\pi_{\alpha}-\frac{1}{2}\pi_{\beta}+\pi_{\beta}\pi_{\alpha}-\frac{1}{2}\pi_{\beta}-\frac{1}{2}\pi_{\alpha}\right)\otimes1\\
 &  & +1\otimes\left(\pi_{\alpha}\pi_{\beta}-\frac{1}{2}\pi_{\alpha}-\frac{1}{2}\pi_{\beta}+\pi_{\beta}\pi_{\alpha}-\frac{1}{2}\pi_{\beta}-\frac{1}{2}\pi_{\alpha}\right)\\
 & = & -\left(\pi_{\alpha}-\pi_{\alpha}\pi_{\beta}-\pi_{\beta}\pi_{\alpha}+\pi_{\beta}\right)\otimes1\\
 &  & -1\otimes\left(\pi_{\alpha}-\pi_{\alpha}\pi_{\beta}-\pi_{\beta}\pi_{\alpha}+\pi_{\beta}\right)\\
 & = & -\pi_{\alpha\pm\beta}\otimes1-1\otimes\pi_{\alpha\pm\beta}.
\end{eqnarray*}
 So one finds that for $\left(\alpha\vert\beta\right)=\mp1$ one has
\begin{eqnarray*}
\left\{ X(\alpha),X(\beta)\right\}  & = & \pi_{\alpha\pm\beta}\otimes\pi_{\alpha\pm\beta}-\pi_{\alpha\pm\beta}\otimes1-1\otimes\pi_{\alpha\pm\beta}+\frac{2}{4}\cdot1\otimes1\\
 & = & \pi_{\alpha\pm\beta}\otimes\pi_{\alpha\pm\beta}-\pi_{\alpha\pm\beta}\otimes1-1\otimes\pi_{\alpha\pm\beta}+\frac{1}{2}\cdot1\otimes1\\
 & = & X\left(\alpha\pm\beta\right)
\end{eqnarray*}
as desired. 
\end{proof}

\begin{rem}
Again note that the canonical Weyl group representation $\rho : W \to \mathrm{GL}(\mathrm{Sym}^2(\mathfrak{h}^*))$ yields $X(\alpha) = \rho(s_\alpha) - \frac{1}{2} \mathrm{id}$. Therefore Remark~\ref{Paul} applies and the statement of Proposition~\ref{free52} in fact follows from the observation that $\rho$ (restricted to any standard subgroup $\mathrm{Sym}_3)$ does not contain the sign representation as an irreducible component.  
\end{rem}

\end{document}